\documentclass[12pt,a4paper]{amsart}
\usepackage[a4paper, left=28mm, right=28mm, top=28mm, bottom=34mm]{geometry}
\usepackage{latexsym}
\usepackage{amsmath}
\usepackage{amssymb}
\usepackage{amsthm}
\usepackage{amscd}
\usepackage{mathrsfs}
\usepackage{mathtools}
\usepackage[all]{xy}
\usepackage{graphicx}
\usepackage{comment}
\usepackage{arydshln}
\usepackage{stmaryrd}
\usepackage{url} 
\usepackage{textcase}
\usepackage[dvipsnames]{xcolor}
\usepackage[normalem]{ulem}

\newtheorem{theorem}{Theorem}[section]
\newtheorem{lemma}[theorem]{Lemma}
\newtheorem{corollary}[theorem]{Corollary}
\newtheorem{proposition}[theorem]{Proposition}

\newtheorem{example}[theorem]{Example}
\theoremstyle{definition}

\newtheorem{remark}[theorem]{Remark}
\newtheorem{definition}[theorem]{Definition}

\def\A{{\mathbb A}}
\def\F{{\mathbb F}}

\def\Z{{\mathbb Z}}
\def\C{{\mathbb C}}
\def\Ql{\overline{\mathbb{Q}}_\ell}

\theoremstyle{remark}

\makeatletter

\@addtoreset{equation}{section}
\makeatother

\makeatletter

\@addtoreset{equation}{section}
\makeatother

\makeatletter
\newcommand{\subsubsubsection}{\@startsection{paragraph}{4}{\z@}%
 {1.0\Cvs \@plus.5\Cdp \@minus.2\Cdp}%
 {.1\Cvs \@plus.3\Cdp}%
 {\reset@font\sffamily\normalsize}
 }
\makeatother
\DeclareMathOperator{\Tr}{Tr}
\DeclareMathOperator{\Nr}{Nr}

\DeclareMathOperator{\Ima}{Im}

\usepackage{bm}

\begin{document}

\title[Additive Polynomial Factorization \& 
 van der Geer--van der Vlugt curves]{Factorization of Additive Polynomials and van der Geer--van der Vlugt curves in characteristic $2$}

\author{Tetsushi Ito}
\address{
Department of Mathematics, Faculty of Science, Kyoto University
Kyoto, 606--8502, Japan}
\email{tetsushi@math.kyoto-u.ac.jp}

\author{Daichi Takeuchi}
\address{Department of Mathematics,
Institute of Science Tokyo,
2-12-1 Ookayama, Meguro-ku, Tokyo, 152-8551, Japan
}
\email{daichi.takeuchi4@gmail.com}

\author{Takahiro Tsushima}
\address{
Keio University School of Medicine,
4-1-1 Hiyoshi, Kohoku-ku,
Yokohama, 223-8521, Japan}
\email{tsushima@keio.jp}

\date{}



\begin{abstract}
In our previous work, we gave a formula for the Frobenius eigenvalues of van der Geer--van der Vlugt curves in characteristic $2$ by considering suitable quotients of the curve. 
Although the formula is explicit, it depends on many choices, which makes the formula complicated.
In this article, we take a different approach using a factorization of additive polynomials, and prove a new formula. The resulting formula is simpler and is useful for explicit computations. 
As applications, we provide a method for constructing maximal and minimal van der Geer--van der Vlugt curves, and show that every such curve arises from this construction. We also compute various examples of van der Geer--van der Vlugt  curves and study their periods. 
\end{abstract}

\maketitle

\section{Introduction}
\label{Introduction}

Let $p_0$ be a prime, $p$ a power of $p_0$, and $q$ a power of $p$. 
Let $R(x)\in\F_q[x]$ be an $\F_p$-linearized polynomial, namely a polynomial of the form 
\[
R(x)=\sum_{i=0}^e a_i x^{p^i} \in \F_q[x].\]
Assume that $e\geq1$ and $a_e \neq 0$. We consider the smooth plane curve over $\F_q$ defined by 
\[
C_R: y^p-y=xR(x). 
\] 
We denote by $\overline{C}_R$ its smooth compactification. We call $C_R$ and $\overline{C}_R$ the \textit{van der Geer--van der Vlugt curves} associated with $R(x)$. 
Explicit computation of the Frobenius eigenvalues of van der Geer--van der Vlugt curves is an interesting problem in number theory and coding theory: see \cite{BP, C, GV, ITT, ITT0, TT}.
In~\cite{ITT} and~\cite{TT}, when $p_0$ is odd, a formula for the Frobenius eigenvalues was given in terms of quadratic Gauss sums. The case $p_0=2$ is excluded there, which reflects the fact that the quadratic Gauss sum is not defined when $p_0=2$. For this reason, the Frobenius eigenvalues of van der Geer--van der Vlugt curves are more mysterious, especially in characteristic $2$. 

From now on, we assume that $p_0=2$. In~\cite{ITT0}, the authors studied the case $p_0=2$ and proved a formula for the Frobenius eigenvalues of van der Geer--van der Vlugt curves in characteristic $2$.
However, the formula proved in~\cite{ITT0} depends on many choices, which makes the formula complicated. 
In this article, we take another approach to the Frobenius eigenvalues and prove a new formula. The resulting formula is simpler and is useful for explicit computations. 
As applications, we determine the maximal and minimal curves among a family of twists of van der Geer--van der Vlugt curves.
We also study the periods and parities of van der Geer--van der Vlugt curves. 

\medskip

We describe our results in more detail. In~\cite{ITT0}, the curve $C_R$ is studied through an action of a Heisenberg group $H_R$. In loc.~cit., we fix a maximal abelian subgroup $A\subset H_R$, and associate an $\F_p$-linearized polynomial $F_A$. The formula proved in \cite{ITT0} is described using such a datum $(A,F_A)$. In this article, we start with an $\F_p$-linearized polynomial $F$, which serves as $F_A$. It turns out that this approach simplifies the computation of the Frobenius eigenvalues, and hence the resulting formula. Note that this computation does not require detailed group-theoretic properties of $H_R$. 

In this paper, we start with a pair $(F,t)$ where 
\[F(x)=\sum_{i=0}^e b_i x^{p^i} \in \F_q[x]
\] 
is an $\F_p$-linearized polynomial and $t\in \F_q$. We further impose the following conditions on $F$. Let $\F$ be an algebraic closure of $\F_q$. For each $x\in \F$, define 
\[
F^\ast(x):=\sum_{i=0}^e (b_i x)^{p^{-i}}\in\F. \]
Such an $F^\ast$ is called the \emph{adjoint} of the $\F_p$-linearized polynomial $F$. 
The adjoint $F^\ast$ defines an $\F_p$-linear map $\F\to \F$, and we denote by $W_F^\ast$ its kernel. 
\begin{enumerate}
    \item We have $e\geq1$ and $b_0b_e\neq0$. 
    \item  We have $F^\ast(1)=0$. 
\item We have $W_F^\ast\subset \F_q$.   
\end{enumerate}
For such a pair $(F,t)$, define 
\begin{align*}
    R_{F}(x):=&\sum_{i=1}^e a_i x^{p^i} \in \F_q[x], \quad 
\textrm{with } a_i:=\sum_{j=0}^{e-i} (b_j b_{j+i})^{p^{-j}},\\
\alpha_F(t):=&F^\ast(t)^2+\sum_{0 \le i<j \le e} b_i^{p^{-i}} b_j^{p^{-j}} \in \F_q.
\end{align*}
We also set 
\[R_{F,t}(x):=R_F(x)+\alpha_F(t)x.\]
We then associate to $(F,t)$ a smooth affine $\F_q$-curve 
\[
D_{F,t} \colon y^p-y=x R_{F,t}(x), 
\]
which is one of the main objects of study in this paper. 
In our previous notation, the curve $D_{F,t}$ is denoted by $C_{R_{F,t}}$. 

If one uses the adjoint operator $f\mapsto f^\ast$ for $\F_p$-linearized polynomials, whose precise definition is recalled in Definition~\ref{defadjoint}, then one can verify that $R_F$ satisfies 
\[
R_F+R_F^\ast=F^\ast\circ F. 
\]
The new insight presented in this article is that, for a given $\F_p$-linearized polynomial $R$, a factorization of the form $R+R^\ast=F^\ast\circ F$ allows one to compute the Frobenius eigenvalues of $C_R$ explicitly.

To some extent, one may say that the curves $D_{F,t}$ form a sufficiently large subclass of the van der Geer--van der Vlugt curves. For instance, every $\F_q$-maximal or $\F_q$-minimal van der Geer--van der Vlugt curve is isomorphic to $D_{F,t}$ for some pair $(F,t)$. See Theorem~\ref{recipe for ex curve intro}. In fact, for a given $R(x)$, we give a necessary and sufficient condition for $R$ to admit a pair $(F,t)$ satisfying  $R=R_{F,t}$. For the precise statement, see Theorem~\ref{detR}; see also Corollary~\ref{VrFq} and Theorem~\ref{mainthm}.

In what follows, we highlight some of the main results of this paper.

\subsection{An Explicit Formula for the Frobenius Eigenvalues of $D_{F,t}$}

Let $(F,t)$ be a pair as above, and let $D_{F,t}$ be the associated curve. Let $\overline{D}_{F, t}$ denote the smooth compactification of $D_{F, t}$. The $L$-polynomial of $\overline{D}_{F,t}$ is described as follows. 

Let $W_2$ denote the Witt vector scheme of length two. 
Let $\ell$ be an odd prime and  
let 
\[
Q_q \colon \F_q \xrightarrow{x \mapsto (x,0)} W_2(\F_q) \xrightarrow{\Tr_{q/2}} W_2(\F_2) \xrightarrow{\xi_2}
\overline{\mathbb{Q}}^{\times}_{\ell}, 
\]
where $\Tr_{q/2} \colon W_2(\F_q) \to W_2(\F_2)$ denotes the trace map and 
$\xi_2$ is a fixed faithful character of the group $W_2(\F_2) \simeq \mathbb{Z}/4\mathbb{Z}$. Set $\sqrt{-1}:=\xi_2(1,0)$, which is a primitive $4$-th root of unity. 

The function $Q_q$ is an even-characteristic analogue of the quadratic form $\frac{1}{2}x^2$, in the sense that 
\[
Q_q(x+y)Q_q(x)^{-1}Q_q(y)^{-1}=(-1)^{\Tr_{\F_q/\F_2}(xy)}. 
\]
\begin{theorem}[Theorem \ref{FEV1}]\label{thm1 intro}
The $L$-polynomial 
\[
L(\overline{D}_{F,t}/\F_q,T):=\det(1-{\rm Fr}_qT \mid H^1(\overline{D}_{F,t,\F},\Ql))
\]
satisfies 
\[
L(\overline{D}_{F,t}/\F_q,T)=\prod_{v\in W_F^\ast}(1-\tau_vT)^{p-1},
\]
where $\tau_v:=Q_q(t+v)^{-1}(-1-\sqrt{-1})^{[\F_q:\F_2]}$. 
\end{theorem}

The above formula is simpler than the one obtained in  \cite[Theorem 1.1]{ITT0}. In loc.~cit., we impose a certain condition on $C_R$. One can show that this condition is equivalent to the existence of a pair $(F,t)$ as above such that $R(x)=R_{F,t}(x)$. See Theorem~\ref{detR} for details.  

\subsection{Maximal and Minimal Twists}
Fix a polynomial 
\[
R(x)=\sum_{i=1}^ea_ix^{p^i}
\]
with $a_i\in\F_q$ and $a_e\neq0$. For each $a\in \F_q$, set $R_a(x):=R(x)+ax$, and consider the van der Geer--van der Vlugt curve $C_a:=C_{R_a}$. 
Its compactification is denoted by $\overline{C}_{a}$.
The curves $C_{a}$, $\overline{C}_{a}$
are called the \textit{twists} of $C_{0}$, $\overline{C}_{0}$, respectively. 
We regard $C_a$ as a family of van der Geer--van der Vlugt curves parametrized by $a$. 

Let $T_{R,\max}$ (resp.\ $T_{R,\min}$) be the set of 
$a \in \F_q$ for which 
the twist $\overline{C}_{a}$ is 
$\F_q$-maximal (resp.\ $\F_q$-minimal). Here, a smooth projective curve over $\F_q$ is said to be \emph{$\F_q$-maximal} (resp.\ \emph{$\F_q$-minimal}) if it attains the Hasse--Weil upper bound (resp.\ lower bound). We also say that $\overline{C}_a$ is \emph{$\F_q$-extremal} if it is either $\F_q$-maximal or $\F_q$-minimal. 
For the precise definition, see Definition~\ref{mmx}. 

For each $x\in \F$, define 
\[
E(x):=R(x)+\sum_{i=1}^e(a_ix)^{p^{-i}}, 
\]
and 
\[V:=\ker(E\colon \F\to \F). \]
It is known that, if either of $T_{R,\max}$ or $T_{R,\min}$ is nonempty, then $V\subset \F_q$. Moreover, the condition $V\subset\F_q$ implies that $[\F_q:\F_p]$ is even. 
For this reason, we henceforth assume that $V\subset \F_q$ and $2\mid[\F_q:\F_p]$. In this case, van der Geer and van der Vlugt \cite[(6.2)]{GV} determine the cardinalities of the sets $T_{R,\max}$ and $T_{R,\min}$, at least when $p=2$. In this paper, we describe these sets by using Theorem~\ref{thm1 intro} as follows. 

Assume that $V\subset\F_q$. Then, by Lemma~\ref{F always exists}, there exists $F$ satisfying the above conditions such that $R=R_F$. Choose and fix such an $F$. 
Let $\Tr_{q/2} \colon \F_q \to \F_2$ denote the trace map and set 
$\psi_q(x):=(-1)^{\Tr_{q/2}(x)}$ for $x \in \F_q$. 
We define 
\begin{align*}
S_{F}:&=\{t \in \F_q \mid \forall v \in W_F^\ast,\ Q_q(v)=\psi_q(tv)\}, \\
S_{F,-}:&=\left\{t \in S_F \mid Q_q(t)=-(\sqrt{-1})^{[\F_q:\F_2]/2}\right\}, \\
S_{F,+}:&=\left\{t \in S_F \mid Q_q(t)=(\sqrt{-1})^{[\F_q:\F_2]/2}\right\}. 
\end{align*}
In the following theorem,
we give a parametrization of maximal and minimal twists.

\begin{theorem}[Theorem \ref{mainthm}]\label{mainthm intro}
We have
\[
\alpha_F(S_{F,-}) = T_{R_F,\max}, \qquad \alpha_F(S_{F,+}) = T_{R_F,\min}. 
\]
\end{theorem}

In other words, the curve $\overline{C}_a$ is $\F_q$-maximal if and only if $a=\alpha_F(t)$ for some $t\in S_{F,-}$. A similar statement holds for $\F_q$-minimal twists. 

In certain cases, $S_{F,-}$ and $S_{F,+}$ admit a more explicit description; see Theorems \ref{easycase} and \ref{exmax}.

\medskip

As a consequence of Theorem~\ref{mainthm intro}, we obtain a recipe for constructing $\F_q$-extremal van der Geer--van der Vlugt curves from linear-algebraic data. Let $p$ be a power of $2$, and $q$ be a power of $p$. Assume that $[\F_q:\F_p]$ is even. 

Consider a pair $(W,t)$, where 
\begin{itemize}
    \item $W$ is an $\F_p$-linear subspace of $\F_q$ such that $1\in W$, 
    \item $t$ is an element of $\F_q$ such that 
    \[
    Q_q(v)=\psi_q(tv)\qquad\text{for all } v\in W. 
    \]
\end{itemize}
To such a pair $(W,t)$, we define 
\[
f_W(x):=\prod_{v\in W}(x-v),\qquad F_W:=(\tau^{-e}f_W)^\ast, 
\]
where $e:=\dim_{\F_p}W$. One can check that $F_W$ satisfies conditions \textup{(i)--(iii)} stated above. 
Let $D_{F_W,t}$ be the associated $\F_q$-curve. 
\begin{theorem}[Theorem~\ref{recipe for ex curve}]\label{recipe for ex curve intro}
The following statements hold: 
\begin{enumerate}
    \item The $\F_q$-curve $D_{F_W,t}$ is $\F_q$-extremal. It is $\F_q$-maximal (resp.\ $\F_q$-minimal) if and only if 
    \[
    Q_q(t)=-(\sqrt{-1})^{[\F_q:\F_2]/2}\qquad\text{(resp.\ }Q_q(t)=(\sqrt{-1})^{[\F_q:\F_2]/2}\text{)}. 
    \]
    \item Every $\F_q$-extremal van der Geer--van der Vlugt curve is obtained in this way. More precisely, let 
\[R(x)=\sum_{i=0}^e a_i x^{p^i} \in 
 \F_q[x],\]
 where $e\geq1$ and $a_e\neq0$, and assume that $C_R$ is $\F_q$-extremal. Then there exists a pair $(W,t)$ as above and $a\in\F_q^\times$ such that $(x,y)\mapsto (ax,y)$ induces an isomorphism 
 \[
 C_R\xrightarrow{\cong}D_{F_W,t}. 
 \]
\end{enumerate}
\end{theorem}

\subsection{Construction of Maximal $\F_{q^2}$-curves}

The following theorem provides an explicit and effective construction of maximal curves from the $\F_p$-linearized polynomial $F$.

\begin{theorem}[Theorem \ref{3}]
Assume $\{x \in \F \mid F^\ast(F(x))=0\} \subset \F_q$. Then, for $t \in \F_q$, the curve $\overline{D}_{F, t}$ is $\F_{q^2}$-maximal if and only if it satisfies $\Tr_{q/2}(t)={([\F_q:\F_2]/2)+1}$. 
\end{theorem}

For examples of van der Geer--van der Vlugt curves that are $\F_q$-maximal, see Subsections~\ref{MEDST} and~\ref{EMT}. In particular, Proposition~\ref{Hermitian curve} treats twists of the Hermitian curve. Another example is provided in Theorem~\ref{exRx}. This example may be viewed as a characteristic two analogue of the examples in \cite[Theorem 1.1]{ITT}.

\subsection{Periods and Parities}\label{Periods and Parities}

For a supersingular curve $C$ over $\F_p$, the smallest positive integer $\mu$ such that $C$ is $\F_{p^\mu}$-maximal or $\F_{p^\mu}$-minimal is called the \textit{$\F_p$-period} of $C$. Its \emph{$\F_p$-parity} $\delta$ is defined to be $-1$ if $C$ is $\F_{p^\mu}$-maximal and $1$ otherwise\footnote{We have chosen a sign convention different from the one used in the literature \cite{KP,SX}. Our convention appears more natural for the following reason. Let $S$ be a normal $\F_p$-scheme of finite type, and let $C$ be a relative van der Geer--van der Vlugt curve over $S$ (see Appendix~\ref{Gal} for a precise definition). 
Assume that, for every closed point $s\in S$, the fiber $C_s$ is $k(s)$-extremal. Then, using Corollary~\ref{T2 dominates T1}, one can show that there exists a quadratic character 
\[
\chi\colon \pi_1(S)^{\rm ab}\to \Ql^\times
\]
such that $\chi({\rm Fr}_s)=1$ if and only if $C_s$ is $s$-minimal, and $\chi({\rm Fr}_s)=-1$ if and only if $C_s$ is $s$-maximal. }.

In Subsection~\ref{periods}, we investigate which pairs $(\mu,\delta)$ can arise as the periods and parities of van der Geer--van der Vlugt curves.  Let 
\[R(x)=\sum_{i=0}^ea_ix^{p^i}\in \F_p[x],\]
where $e\geq1$ and $a_e\neq0$. Let $\overline{C}_R$ denote the smooth compactification of the associated curve $C_R$. 
Let $\mu_R$ and $\delta_R$ denote the period and parity of $\overline{C}_R$, respectively.

The results in Subsection~\ref{periods} are summarized below. 
\begin{itemize}
    \item The period $\mu_R$ is always even. 
    \item For every even integer $\mu\geq2$, the polynomial  $R(x):=x^{p^{\mu/2}}$ satisfies 
    \[
    (\mu_R,\delta_R)=(\mu,-1). 
    \]
    \item When $p=2$, for any $R$, we have 
    \[
    (\mu_R,\delta_R)\neq(4,1). 
    \]
    \item When $p>2$, the polynomial $R(x):=x^p+ax$ with $a\in\F_p^\times$ and $\Tr_{p/2}(a)=0$ satisfies 
    \[
    (\mu_R,\delta_R)=(4,1). 
    \]
    \item Let $\mu$ be a positive multiple of $4$. Then the polynomial 
    \[
    R(x):=\sum_{i=1}^{\mu/4}x^{p^{2i-1}}
    \]
    satisfies 
    \[
    (\mu_R,\delta_R)=(\mu,1). 
    \]
\end{itemize}

Hence, the remaining case is when $2\mid \mu$, $4\nmid \mu$, and $\delta=1$. 
In this situation, we have the following result.
\begin{proposition}[Corollary \ref{lasc}]
Let $m\geq1$ be an odd integer. Assume that $\F_p$ contains a primitive $m$-th root of unity. Then, for any $R$, 
\[
(\mu_R,\delta_R)\neq(2m,1). 
\]
\end{proposition}

\medskip

Finally, in Appendix \ref{Gal}, we discuss relative van der Geer--van der Vlugt curves over a base scheme $S$. Theorem \ref{T_12} relates the maximality of the fiber over a closed point $s\in S$ to the splitting behavior of the geometric Frobenius element at $s$. 

\subsection*{Organization}

The organization of this paper is as follows.
In Section \ref{linearized}, we introduce notation on linearized polynomials.
We study the relationship between the kernels of a linearized polynomial and its adjoint.
In Section \ref{langtorsor}, we study $\ell$-adic sheaves coming from the Lang torsor over the group scheme $W_2$ of Witt vectors of length two.
In Section \ref{CFE}, we calculate the Frobenius eigenvalues explicitly.
In Section \ref{max2}, we study maximality or minimality of twists of van der Geer--van der Vlugt curves.
Then, in Section \ref{constrmaxcurve}, we construct several examples  using our formulas. 
In Appendix \ref{Gal}, we relate maximality to the splitting behavior of the geometric Frobenius element in the relative setting.

\subsection*{Notation}
Throughout the paper,
we fix a prime number $p_0$. 
We fix an algebraic closure $\mathbb{F}$ of $\mathbb{F}_{p_0}$, and for every power $q$ of $p_0$, we write $\F_q$ for the subfield of $\F$ with $q$ elements. We denote by ${\rm Fr}_q$ the geometric Frobenius element of ${\rm Gal}(\F/\F_q)$. 

Let $p$ be a power of $p_0$, and $q$ a power of $p$. We write $\Tr_{q/p}$ for the trace map 
$\mathrm{Tr}_{\mathbb{F}_{q}/\mathbb{F}_p} \colon \mathbb{F}_{q} \to \mathbb{F}_p$. 
We  use the same notation for the trace map 
\[
\Tr_{q/p}\colon W_2(\F_q)\to W_2(\F_p),\quad (x,y)\mapsto \sum_{i=0}^{n-1}(x^{p^i},y^{p^i}), 
\]
where $q=p^n$. 

In most of the paper, we consider the case $p_0 = 2$. 
We fix a prime number $\ell\neq p_0$, and fix an algebraic closure $\Ql$ of $\mathbb{Q}_\ell$.

\section{Computations on $\F_p$-Linearized Polynomials}
\label{linearized}

Let $p_0$ be a prime number, and let $p$ be a power of $p_0$. Let $k$ be a perfect field containing $\F_p$. 
Let  $k\{\tau^{\pm1}\}$ denote the $k$-algebra 
\[\bigoplus_{i\in\mathbb{Z}}k\tau^i,\]
whose multiplication is determined by 
\[(a\tau^i)\cdot (b\tau^j)=ab^{p^i}\tau^{i+j} \text{ \hspace{3mm} for $i,j\in \mathbb{Z}$ and $a,b\in k$.}\] 
For $f=\sum_ia_i\tau^i\in k\{\tau^{\pm1}\}$, we define the corresponding polynomial 
\[
f(x):=\sum_ia_ix^{p^i}\in k[x^{p^{-\infty}}],  \]
where $k[x^{p^{-\infty}}]$ denotes the perfection of $k[x]$. 
Conversely, given a polynomial of the form $f(x)=\sum_ia_ix^{p^i}\in k[x^{p^{-\infty}}]$, we write $f$ for the corresponding element $\sum_ia_i\tau^i\in k\{\tau^{\pm1}\}$.

A polynomial of the form $\sum_{i\geq0}a_ix^{p^i}$ is called an \emph{$\F_p$-linearized polynomial}. 
By abuse of terminology, we also call the corresponding element $\sum_{i\geq0}a_i\tau^i\in k\{\tau^{\pm1}\}$ an $\F_p$-linearized polynomial. 

An $\F_p$-linearized polynomial $f(x)=\sum_{i\geq0}a_ix^{p^i}$ is said to be \emph{separable} if $a_0\neq0$. Equivalently, all roots of $f(x)$ are simple.

Let $\overline{k}$ be an algebraic closure of $k$. To an element $f=\sum_ia_i\tau^i\in k\{\tau^{\pm1}\}$, we associate the $\F_p$-linear map 
\[\overline{k}\to \overline{k},\quad x\mapsto f(x)=\sum_ia_ix^{p^i}.
\]
By abuse of notation, we also denote this linear map by $f$.  This construction defines a ring homomorphism $k\{\tau^{\pm1}\}\to{\rm End}_{\F_p}(\overline{k})$. We write $\ker f$ for the kernel of the associated map $\overline{k}\to \overline{k}$. An element of the kernel is called a \emph{root} of $f$.

We recall some basic results. 
Let ${\mathcal V}$ denote the set of finite-dimensional $\F_p$-subspaces of $\overline{k}$. For $W\in\mathcal V$, define  
\[f_W(x):=\prod_{v\in W}(x-v).\]
Then $f_W(x)$ is a monic separable $\F_p$-linearized polynomial with $\ker f_W=W$. Furthermore, the element $f_W\in k\{\tau^{\pm1}\}$ is uniquely characterized by these properties. 

\begin{lemma}\label{basic}
The following statements hold: 
\begin{enumerate}
\item The map 
\[\overline{k}^\times\times\mathbb{Z}\times{\mathcal V}\to \overline{k}\{\tau^{\pm1}\}\setminus\{0\},\qquad (a,n,W) \to a\tau^nf_W \]
 is bijective. In particular, for every nonzero $f\in \overline{k}\{\tau^{\pm1}\}$, there exists a unique triple $(a,n,W)$ such that 
 \[
 f=a\tau^nf_W. 
 \]
 Moreover, we have $\ker f=W$. 
 
\item Let $f=\sum_{i=r}^sa_i\tau^i$ with $a_ra_s\neq0$. Then 
\[\dim_{\F_p}\ker f=s-r. \]
\item For $f,g\in\overline{k}\{\tau^{\pm1}\}$, the following conditions are equivalent: 
\begin{enumerate}
\item There exists $h\in\overline{k}\{\tau^{\pm1}\}$ such that $f=hg$. 
\item We have $\ker g\subset \ker f$. 
\end{enumerate}
\end{enumerate}
\end{lemma}
\proof{
We first prove (i).  It is clear that $a\tau^nf_W$ lies in 
$\overline{k}\{\tau^{\pm1}\}\setminus\{0\}$. We construct the inverse map. Take a nonzero element $f\in \overline{k}\{\tau^{\pm1}\}$, and write $f=\sum_{i=r}^sa_i\tau^i$ with $r\leq s$ and $a_r,a_s\neq0$. Then $f$ can be written uniquely as  
\[f=a_s\tau^r\sum_{i=0}^{s-r}b_i\tau^i,\]
where $b_0\neq0$ and $b_{s-r}=1$. 

Let $W=\ker f$. Since $a_s\tau^r$ acts as an automorphism, $W$ coincides with the kernel of  $\sum_{i=0}^{s-r}b_ix^{p^i}$. 
Since this polynomial is monic and separable, the uniqueness of $f_W$ implies that 
\[\sum_{i=0}^{s-r}b_ix^{p^i}=f_W(x). \]
Therefore, the assignment $f\mapsto (a_s,r,W)$ defines the inverse map. The same argument also proves (ii).

We now prove (iii).  The implication $\text{(a)}\Rightarrow\text{(b)}$ is clear, since $\overline{k}\{\tau^{\pm1}\}\to{\rm End}_{\F_p}(\overline{k})$ is a ring homomorphism. We prove $\text{(b)}\Rightarrow\text{(a)}$. If $f=0$, then we may  take $h=0$. Assume now that $f\neq0$. Set $W_f:=\ker f$ and $W_g:=\ker g$, respectively. Since $W_g\subset W_f$ and $f\neq0$, the space $W_g$ is finite-dimensional. In particular, we have $g\neq0$. Therefore, we may write 
\[
f=\alpha_f f_{W_f},\quad g=\alpha_g f_{W_g}
\]
where $\alpha_f,\alpha_g$ are units in $\overline{k}\{\tau^{\pm1}\}$. 

Let $v_1,\dots,v_r\in W_f$ be representatives of the quotient $W_f/W_g$. 
Then 
\[f_{W_f}(x)=\prod_{i=1}^r(f_{W_g}(x)-f_{W_g}(v_i)).\]
Hence, $f_{W_f}=h_1f_{W_g}$ for some $h_1\in \overline{k}\{\tau^{\pm1}\}$. Since $\alpha_f$ and $\alpha_g$ are units, it follows that
\[
f=hg
\]
for some $h\in \overline{k}\{\tau^{\pm1}\}$. 
\qed}

Assume that $f\in k\{\tau^{\pm1}\}$ is of the form $\sum_{i\geq0}a_i\tau^i$. Then $f$ induces a $k$-morphism 
\[\A^1_k\to\A^1_k,\]
which we denote by the same symbol $f$. Note that the $\F_p$-linear map $f\colon \overline{k}\to \overline{k}$ considered above coincides with the map on the $\overline{k}$-valued points induced by the morphism $ f\colon \A_k^1\to \A_k^1$. 
\begin{lemma}\label{b2}
The following statements hold: 
\begin{enumerate}
\item Let
\[
f=\sum_{i\ge0} a_i\tau^i \in k\{\tau^{\pm1}\}
\]
with $a_0\ne0$. 
\begin{enumerate}
\item The morphism $f\colon \A^1_k\to \A^1_k$ is finite \'etale. 
\item The morphism $f$ is a Galois covering if and only if 
$\ker f\subset k$. In this case, $f$ is an abelian Galois covering with Galois group $\ker f$.
\end{enumerate} 

\item Let $f_1,f_2,f_3$ be nonzero elements of $k\{\tau^{\pm1}\}$, and put $f=f_1f_2f_3$. Assume  that  $\ker f\subset k$. Then $\ker f_2\subset k$. 
\end{enumerate}
\end{lemma}
\proof{
(i)(a) Finiteness is immediate. \'Etaleness follows from  the Jacobian criterion. 

\medskip

\noindent
(i)(b) Let $f_{\overline{k}}\colon \A^1_{\overline{k}}\to \A^1_{\overline{k}}$ be the base change of $f$. For each $a\in \ker f$, consider the automorphism of $\A^1_{\overline{k}}$ given by $x\mapsto x+a$. Since  $f_{\overline{k}}(x+a)=f_{\overline{k}}(x)$, we obtain a group homomorphism \[\ker f\to {\rm Aut}(f_{\overline{k}}).
\]
This homomorphism is injective. 
Since the degree of $f_{\overline{k}}$ is equal to the cardinality of $\ker f$, it follows that the above map is an isomorphism. Hence $f_{\overline{k}}$ is an abelian Galois covering with Galois group $\ker f$. 

Therefore, $f$ is Galois over $k$ if and only if each automorphism $x\mapsto x+a$  is defined over $k$, which is equivalent to $\ker f\subset k$. This proves the assertion. 

\medskip

\noindent
(ii) For each $i=1,2,3$, choose an integer $n_i$ such that $\tau^{n_i}f_i$ defines a separable $\F_p$-linearized polynomial. Put 
\[\widetilde{f}:=\tau^{n_1+n_2+n_3}f.\]
Then $\widetilde{f}$ also defines a separable $\F_p$-linearized polynomial. By (i)(b), the $k$-morphism $\widetilde{f}\colon \A^1_k\to\A^1_k$ is an abelian Galois covering. Since 
\[
\widetilde{f}=(\tau^{n_1+n_2+n_3}f_1\tau^{-n_2-n_3})\cdot(\tau^{n_3+n_2}f_2\tau^{-n_3})\cdot(\tau^{n_3}f_3), 
\]
the morphism 
\[g:=\tau^{n_3+n_2}f_2\tau^{-n_3}\colon \A^1_k\to\A^1_k
\]
is an intermediate covering of $\tilde{f}$. Hence it is also a Galois covering. Applying (i)(b) again, we obtain 
$\ker g\subset k$. 
Since 
\[\ker f_2=(\ker g)^{p^{-n_3}}:=\{a\in \bar{k}\mid a^{p^{n_3}}\in \ker g\},\]
and since $k$ is perfect, the assertion follows. 
\qed}
\begin{definition}\label{defadjoint}
For $f=\sum_ia_i\tau^i\in k\{\tau^{\pm1}\}$, define  
\[f^\ast \coloneqq \sum_ia_i^{p^{-i}}\tau^{-i}. \] 
This element is called the \emph{adjoint} of $f$. 
\end{definition}
The adjoint defines an anti-involution on
$k\{\tau^{\pm1}\}$; namely,
\[
(fg)^\ast=g^\ast f^\ast,
\qquad
(f^\ast)^\ast=f.
\]
The following observation will be useful. 
\begin{lemma}\label{b4}
For an $\F_p$-algebra $R$ and elements $a,b\in R$, write $a\sim b$ if $a-b=c^p-c$ for some $c\in R$. 

Let $f\in k\{\tau^{\pm1}\}$. Then 
\[xf(y)\sim yf^\ast(x)\]
in $k[x^{p^{-\infty}}, y^{p^{-\infty}}]$. 
\end{lemma}
\proof{Let $f=\sum_ia_i\tau^i$. Then 
\[xf(y)-yf^\ast(x)=\sum_i(a_ixy^{p^i}-y(a_ix)^{p^{-i}}). \]
Set $z_i=y(a_ix)^{p^{-i}}$. Then we have  
\[a_ixy^{p^i}-y(a_ix)^{p^{-i}}=z_i^{p^i}-z_i.\]
Hence 
\[a_ixy^{p^i}\sim y(a_ix)^{p^{-i}}\]
for every $i$. 
Summing over $i$, the assertion follows. 
\qed}


\begin{lemma}\label{b3}
For every nonzero element $F \in k\{\tau^{\pm1}\}$, 
there exists a unique element $g(x,y)\in k[x^{p^{-\infty}},y^{p^{-\infty}}]$ satisfying $g(0,0)=0$ and 
\begin{equation}\label{a1}
g(x,y)^p-g(x,y)=xF(y)-yF^\ast(x). 
\end{equation}
Furthermore, $g(x,y)$ satisfies 
\begin{align*}
   g(x+z,y)&=g(x,y)+g(z,y), \\
g(x,y+z)&=g(x,y)+g(x,z), \\
g(ax,y)&=ag(x,y)=g(x,ay) \qquad \text{for all } a\in\F_p. 
\end{align*}
If moreover \(F=F^\ast\), then
\[
g(x,y)=-g(y,x).
\]
\end{lemma}
\proof{
We first show uniqueness. Let $h(x,y)$ be another element satisfying the same conditions. Then we have 
\[(g-h)^p-(g-h)=0\]
in $k[x^{p^{-\infty}},y^{p^{-\infty}}]$. Hence \(g-h\) must belong to \(\F_p\). Since
\[
g(0,0)-h(0,0)=0,
\]
we conclude that \(g=h\). 

We now prove existence. Since $xF(y)\sim yF^\ast(x)$ by Lemma \ref{b4}, we can write 
\[g_1(x,y)^p-g_1(x,y)=xF(y)-yF^\ast(x)\]
for some $g_1 \in k[x^{p^{-\infty}},y^{p^{-\infty}}]$. Evaluating at $x=y=0$, we obtain $g_1(0,0)^p-g_1(0,0)=0$. Therefore, $g(x,y):=g_1(x,y)-g_1(0,0)$ satisfies \eqref{a1} and $g(0,0)=0$. 

\medskip 

\noindent
We prove $g(x+z,y)=g(x,y)+g(z,y)$;  the other identities are proved similarly. Set $h(x,y,z):=g(x+z,y)-g(x,y)-g(z,y)$. Then 
\[h^p-h=(x+z)F(y)-yF^\ast(x+z)-(xF(y)-yF^\ast(x))-(zF(y)-yF^\ast(z))=0. \]
Since $h(0,0,0)=0$, it follows that $h=0$. 
\qed}

Let $F\in k\{\tau^{\pm1}\}$ be a nonzero element. In Lemma \ref{b5} below, we recall results from \cite[4.14]{Go}, stated in a form suitable for our purposes. Set 
\[
W:=\ker F,\qquad W^\ast:=\ker F^\ast. 
\]
By Lemma \ref{basic}(ii), these $\F_p$-vector spaces have the same dimension. 
\begin{lemma}\label{b5}
Let $g(x,y)$ be as in Lemma \ref{b3}. 
For $u\in W$ and $u^\ast\in W^\ast$, define 
\[\omega(u,u^\ast):=g(u^\ast,u). \]
Then the following hold: 
\begin{enumerate}
\item We have $\omega(u,u^\ast)\in\F_p$. 
\item The assignment $(u,u^\ast)\mapsto \omega(u,u^\ast)$ defines a nondegenerate pairing 
\[\omega\colon W\times W^\ast\to \F_p. \]
\item Assume further that $F=F^\ast$, and hence $W=W^\ast$. Then $\omega$ is a symplectic form on $W$. 
\item Let $G:={\rm Gal}(\overline{k}/k)$. Then for every $u\in W$, $u^\ast\in W^\ast$, and $\sigma\in G$, we have 
\[\omega(\sigma u,\sigma u^\ast)=\omega(u,u^\ast). \]
\end{enumerate}
\end{lemma}
\proof{(i) The assertion follows from the identity  
\[\omega(u,u^\ast)^p-\omega(u,u^\ast)=u^\ast F(u)-uF^\ast(u^\ast)=0.\]

(ii), (iii)  The assertions other than nondegeneracy follow from Lemma \ref{b3}. 

We prove nondegeneracy. Let $u^\ast\in W^\ast$ be nonzero. Assume that the polynomial $g(u^\ast,y)$ vanishes on $W$. Then, by Lemma~\ref{basic}(iii), there exists $h\in \overline{k}[y^{p^{-\infty}}]$ such that 
\[g(u^\ast,y)=hF(y).\]
 Substituting $x=u^\ast$ into \eqref{a1}, we obtain  
\[
h^pF(y)^p-hF(y)=u^\ast F(y). \]
Since $F(y)\neq0$, dividing by \(F(y)\) yields  
\[h^pF(y)^{p-1}-h=u^\ast.\]

Suppose that $h$ is non-constant. Since $\overline{k}$ is algebraically closed, there exists $a\in\overline{k}$ such that $h(a)=0$. Evaluating at $y=a$ yields $0=u^\ast$, a contradiction. Therefore, $h$ must be constant. Moreover, $h\neq0$. It follows that 
\[F(y)^{p-1}=h^{-p}(h+u^\ast),\]
and hence $F(y)$ is constant. 
Consequently, $F^\ast$ is a nonzero constant, 
which contradicts the assumption that $W^\ast$ contains a nonzero element $u^\ast$. 

Therefore, the linear map $\omega(-,u^\ast)\colon W\to \F_p$ is nonzero. Since $W$ and $W^\ast$ have the same dimension, it follows that $\omega$ is nondegenerate. 

(iv) Since $g(x,y)$ has coefficients in $k$, we have 
\[
\sigma\bigl(g(u^\ast,u)\bigr)=g(\sigma u^\ast,\sigma u) 
\]
for every $\sigma\in G$. 
The assertion now follows from (i), since $G$ acts trivially on $\F_p$. 
\qed}

\begin{remark}\label{omega in ITT0}
Let $R(x)=\sum_{i=0}^ea_ix^{p^i}$ be an $\F_p$-linearized polynomial in $k[x]$, and set 
\[E:=R+R^\ast,\qquad V:=\ker E. 
\]
By Lemma~\ref{b6}(iii), the vector space $V$ carries a symplectic form $\omega$. 

On the other hand, in \cite{ITT0} and \cite{TT}, the authors consider an $\F_p$-vector space $V_R$ equipped with a symplectic form $\omega_R$. 
The relation between $(V,\omega)$ and $(V_R,\omega_R)$ is as follows: 
\[
V_R=V,\qquad \omega_R(x,y)=\omega(x,y)^{p^e}\quad\text{for all }x,y\in V=V_R. 
\]
Indeed, $V_R$ is defined as the kernel of 
\[
E_R:=R(x)^{p^e}+\sum_{i=0}^e (a_ix)^{p^{e-i}}. 
\]
 Since $E_R=\tau^e E$, we obtain $V=V_R$. 

 Furthermore, by \cite[(2.5)]{ITT0}, the pairing $\omega_R$ is defined by the unique polynomial $\omega_R(x,y)\in\F_q[x,y]$ satisfying 
\[
\omega_R(0,0)=0,\qquad \omega_R(x,y)^p-\omega_R(x,y)=y^{p^e}E_R(x)-x^{p^e}E_R(y). 
\]
On the other hand, the pairing $\omega$ is defined by the unique element $\omega(x,y)\in\F_q[x^{p^{-\infty}},y^{p^{-\infty}}]$ satisfying 
\[
\omega(0,0)=0,\qquad \omega(x,y)^p-\omega(x,y)=yE(x)-xE(y). 
\]
Since $E_R=\tau^e E$, uniqueness implies the equality $\omega_R(x,y)=\omega(x,y)^{p^e}$. 
\end{remark}

\begin{lemma}\label{kfin}
Assume that $k$ is a finite field containing $\F_p$. 
\begin{enumerate}
\item The $\F_p$-vector spaces $W\cap k$ and $W^\ast\cap  k$ have the same dimension. 
\item For an $\F_p$-vector space $V$, set $V^{\vee}:={\rm Hom}_{\F_p}(V,\F_p)$. Then the kernel of the surjection \[k\to (W\cap k)^{\vee}, \quad a\mapsto \Tr_{k/\F_p}(a-)|_{W\cap k}\]
is equal to $F^\ast(k):=\{F^\ast(x) \mid x \in k\}$. 
\end{enumerate}
\end{lemma}
\proof{
(i) 
By Lemma \ref{b5}, the pairing $\omega$ induces a $G$-equivariant isomorphism $W^\ast\to W^\vee$. Taking $G$-invariants, we obtain 
\[W^\ast\cap k\cong (W_G)^\vee,\]
where $W_G$ denotes the coinvariants of $W$. 

It therefore suffices to show that $|W_G|=|W\cap k|$. Let ${\rm Fr}_k\in G$ be the geometric Frobenius element. Since $G$ is topologically generated by ${\rm Fr}_k$, we have an  exact sequence 
\[0\to W\cap k\to W\xrightarrow{{\rm Fr}_k-1}W\to W_G\to 0. \]
This proves the assertion. 

\medskip 

\noindent
(ii) Let $K$ denote the kernel of $k\to (W\cap k)^\vee$. 
 By Lemma \ref{b4}, for every $a\in k$ and $u\in W\cap k$, we have 
\[F^\ast(a)u\sim aF(u)=0. \]
Applying $\Tr_{k/\F_p}$, we obtain $\Tr_{k/\F_p}(F^\ast(a)u)=0$. Hence $F^\ast(k)\subset K$. 

Thus it remains to show that $|F^\ast(k)|=|K|$. Since the map $k\to (W\cap k)^\vee$ is surjective, we have $|K|=|k|/|W\cap k|$. On the other hand, we have a short exact sequence 
\[0\to W^\ast\cap k\to k\xrightarrow{F^\ast} F^\ast(k)\to0.\]
The assertion now follows from (i). 
\qed}

We return to the situation where $k$ is a general perfect field. 
For an $\F_p$-vector subspace $X\subset W$, define its orthogonal complement $X^\perp\subset W^\ast$ by 
\begin{equation}\label{b7}
X^\perp:=\{u^\ast\in W^\ast\mid  \omega(u,u^\ast)=0\text{ for all } u\in X\}.
\end{equation}

\begin{proposition}
\label{b6}
Let the notation be as above, and let $X$ be an $\F_p$-vector subspace of $W$. Let $f$ be an element of $\overline{k}\{\tau^{\pm1}\}$ satisfying $\ker f=X$. Then the following statements hold: 
\begin{enumerate}
\item There exists $h\in \bar{k}\{\tau^{\pm1}\}$ such that $F=hf$. 
\item We have $\ker( h^\ast)=X^\perp$. 
\end{enumerate}
\end{proposition}

\begin{proof}
Part (i) follows from Lemma \ref{basic}(iii). 

We now prove (ii). Since $F^\ast=f^\ast h^\ast$, we have $\ker h^\ast\subset W^\ast$ . Therefore, it suffices to show that, for every $u^\ast\in W^\ast$, the condition $h^\ast(u)=0$ is equivalent to $u^\ast\in X^\perp$. 

We first consider the case where $\dim X=1$. Choose a nonzero vector $v\in X$. Since 
\[\ker f=X=\F_p  v,\]
Lemma~\ref{basic}\textup{(i)} implies that
\[
f=a(\tau-1)v^{-1}
\]
for some unit $a\in k\{\tau^{\pm1}\}$. Hence $Fv=ha(\tau-1)$.  Taking adjoints, we obtain 
\[vF^\ast=(\tau-1)bh^\ast,\]
where $b:=-\tau^{-1}a^\ast$ is also a unit. Therefore,  
\[vF^\ast(x)=(bh^\ast(x))^p-bh^\ast(x). \]
On the other hand, substituting $y=v$ into \eqref{a1}, we obtain
\[
g(x,v)^p-g(x,v)
=
-vF^\ast(x),
\]
since $F(v)=0$. Thus we obtain 
\[
bh^\ast(x)=-g(x,v).
\]
Since $b\colon \overline{k}\to \overline{k}$ is an isomorphism, it follows that for $u^\ast\in W^\ast$, the condition $h^\ast(u^\ast)=0$ is equivalent to $\omega(v,u^\ast)=0$. This proves the assertion when $\dim_{\F_p} X=1$.

\medskip 

\noindent
We now treat the general case. Let $X_1:=\ker(h^\ast)$. We already know that $X_1\subset W^\ast$. We will show that $X_1=X^\perp$. 

For each nonzero element $v\in X$, we set $f_v:=(\tau-1)v^{-1}$ and write $f=h_vf_v$. Then we have $F=hh_vf_v$. By the one-dimensional case, we have 
\[X_1\subset \ker (hh_v)^\ast=\langle v\rangle^\perp.\]
Taking the intersection over all $v\in X\setminus\{0\}$, 
we obtain
\[
X_1\subset X^\perp.
\]
To prove equality, it suffices to compare dimensions.
Since $\omega$ is nondegenerate, we have $\dim X^\perp=\dim W-\dim X$. On the other hand, since $F=hf$, Lemma \ref{basic}(ii) implies that $\dim X_1+\dim X=\dim W$. This completes the proof.
\end{proof}

The proposition has the following corollary. Let  $E\in k\{\tau^{\pm1}\}$ satisfy $E=E^\ast$. By definition of the adjoint, we see that $E$ is of the form 
\begin{equation}\label{a}
a_0\tau^0+\sum_{i=1}^e\bigl(a_i\tau^i+a_i^{p^{-i}}\tau^{-i}\bigr), 
\end{equation}
where $a_i\in k$. Assume that $e\geq1$ and $a_e\neq0$. Let $V:=\ker E$. 
\begin{corollary}\label{dec2}
Let the notation and assumptions be as above. Let $W\subset V$ be a maximal totally isotropic $\F_p$-vector subspace with respect to the symplectic pairing $\omega$ defined in Lemma~\ref{b5}.  Let $F(x)=\sum_{i=0}^eb_ix^{p^i}\in\overline{k}[x]$ be a separable $\mathbb{F}_p$-linearized polynomial with $\ker F=W$. Then there exists a unique constant $a\in \overline{k}^\times$ satisfying 
\[E(x)=F^\ast(aF(x)).\]
If $F(x)\in k[x]$, then $a\in k^\times$. 

If $p_0=2$, then one can choose $F$ so that \[E(x)=F^\ast(F(x)).\]
Moreover, for such an $F$, we have $a_0=F^\ast(1)^2$. 
\end{corollary}
\proof{Write $E=hF$ in $\overline{k}\{\tau^{\pm1}\}$. By Lemma \ref{basic}(i) and Proposition \ref{b6}, the adjoint $h^\ast$ can be written uniquely in the form 
\[h^\ast=a\tau^nF\]
for some $n\in\Z$ and $a\in \overline{k}^\times$. We claim that $n=0$. Indeed, we have 
\[E=hF=F^\ast\tau^{-n}aF.\]
Taking adjoints, we also obtain $E=F^\ast a\tau^n F$. Comparing the highest degrees in the equality $F^\ast\tau^{-n}aF=F^\ast a\tau^n F$, we conclude that $n=0$. 

Therefore, $h=(aF)^\ast$, and hence $E=F^\ast aF$. The uniqueness of $a$ follows from the uniqueness of the above factorization. 

Assume now that $F\in k\{\tau^{\pm1}\}$. Then, for every $\sigma\in G$, we obtain 
\[
E=F^\ast\sigma(a) F. 
\]
By uniqueness, $a=\sigma(a)$. Since this holds for every $\sigma$, it follows that $a\in k^\times$. 

\medskip 

\noindent
Assume that $p_0=2$. Let $b\in k^\times$ and set  $F_1=b^{-1}F$. Then 
\[E=(bF_1)^\ast abF_1=F_1^\ast ab^2F_1. \]
Thus, choosing $b=a^{-1/2}$, we obtain $E=F_1^\ast F_1$. 

Now assume that $F$ satisfies $E=F^\ast F$.  Write $F(x)=\sum_{i=0}^eb_ix^{p^i}$. Then 
\[E(x)=\sum_{i=0}^e(b_iF(x))^{p^{-i}}=\sum_{i,j=0}^e(b_ib_j)^{p^{-i}}x^{p^{j-i}}. \]
Comparing the coefficients of $x$, we obtain $a_0= F^\ast(1)^2$, which completes the proof. 
\qed}

\section{Lang Torsor over $W_2$} 
\label{langtorsor}

In this section, we assume that $p_0=2$. Recall that, for an $\F_2$-algebra $R$, the ring $W_2(R)$ of Witt vectors of length two is defined to be the set $R^2$  whose ring structure is given by 
\begin{align}\label{add}
(a,b)+(c,d)&=(a+c,b+d+ac), \\\notag(a,b) \cdot (c,d)&=(ac,a^2d+c^2b). 
\end{align}
Let $p$ be a power of $2$. 
Let $W_2$ be the commutative group $\F_p$-scheme $\A^2_{\F_p}$ whose addition is given by \eqref{add}. 
Consider the Lang morphism  
\[L_p\colon W_2\to W_2\]
defined by $L_p(x,y):=(x^p,y^p)-(x,y)=(x^p+x,y^p+y+x^2+x^{p+1})$. 
\begin{lemma}\label{Lp}
The morphism $L_p$ is finite \'etale of degree $p^2$. It is a Galois covering with Galois group $W_2(\F_p)$. 
\end{lemma}
\proof{The finiteness of $L_p$ is clear. By the Jacobian criterion, $L_p$ is smooth; hence it is finite \'etale. To determine the degree, we compute the fiber $L_p^{-1}(0,0)$, which is equal to $W_2(\F_p)$. Thus the degree is  $p^2$. 

We define an action of  $ W_2(\F_p)$ on $W_2$ by 
\[W_2(\F_p)\times W_2\to W_2, \quad (\alpha,\beta)\mapsto \alpha+\beta. \]
Since $L_p(\alpha)=0$ for every $\alpha\in W_2(\F_p)$, the morphism $L_p$ is invariant under this action. This induces an injective homomorphism $W_2(\F_p)\to {\rm Aut}(L_p)$. Since  $L_p$ has degree $p^2$, this is an isomorphism. 
\qed}

Let $\ell$ be a prime number different from $2$. Let $W_2(\F_p)^\vee$ denote the character group ${\rm Hom}(W_2(\F_p),\Ql^\times)$. For $\xi\in W_2(\F_p)^\vee$, let ${\mathcal L}_\xi$ be the rank one smooth $\Ql$-sheaf on $W_2$ associated with the $W_2(\F_p)$-torsor $L_p$ and the character $\xi$. 
\begin{lemma}\label{pLp}
There is an isomorphism of sheaves on $W_2$
\[
L_{p\ast}\Ql\cong \bigoplus_{\xi\in W_2(\F_p)^\vee}{\mathcal L}_\xi. \]
\end{lemma}
\proof{
It follows from Lemma \ref{Lp}.\qed}

For a morphism $(f,g)\colon X\to W_2$ of schemes, we denote by ${\mathcal L}_{\xi}(f,g)$ the pullback $(f,g)^\ast {\mathcal L}_{\xi}$. 
\begin{lemma}\label{Lxi1}
Let $X$ be an $\F_p$-scheme. 
Then the following hold:
\begin{enumerate}
\item For two morphisms $(f,g),\, (f',g')\colon X\to W_2$, we have 
\[{\mathcal L}_{\xi}(f,g)\otimes {\mathcal L}_{\xi}(f',g')\cong {\mathcal L}_{\xi}(f+f',g+g'+ff'). \]
\item For a morphism $(f,g)\colon X\to W_2$, we have 
\[{\mathcal L}_{\xi}(f,g)\cong {\mathcal L}_{\xi}(f^p,g^p). \]
\end{enumerate}
\end{lemma}
\proof{
(i) It suffices to prove the assertion in the universal situation, namely, when 
\[X=W_2\times W_2, \quad(f,g)={\rm pr}_1, \quad (f',g')={\rm pr}_2. \]
We show that 
\begin{equation}\label{pr1+pr2}
    {\mathcal L}_{\xi}({\rm pr}_1)\otimes {\mathcal L}_{\xi}({\rm pr}_2)\cong {\mathcal L}_{\xi}({\rm pr}_1+{\rm pr}_2). 
\end{equation}
The sheaf on the left-hand side is associated with the $W_2(\F_p)\times W_2(\F_p)$-torsor 
\[L_p\times L_p\colon W_2\times W_2\to W_2\times W_2\]
together with the character 
\[
W_2(\F_p)\times W_2(\F_p)\to \Ql^\times,\quad (\alpha,\beta)\mapsto \xi(\alpha)\cdot \xi(\beta)=\xi(\alpha+\beta). 
\]
On the other hand, since $L_p$ is additive, we have a commutative diagram 
\[\xymatrix{
W_2\times W_2\ar[d]\ar[rr]^-{L_p\times L_p}&& W_2\times W_2\ar[d]\\
W_2\ar[rr]^-{L_p}& & W_2,}\]
where the vertical arrows are the addition morphisms. Moreover, the left vertical arrow is equivariant with respect to the group homomorphism 
\[
W_2(\F_p)\times W_2(\F_p)\to W_2(\F_p),\quad(\alpha,\beta)\mapsto \alpha+\beta. 
\]
This proves \eqref{pr1+pr2}. 

\smallskip

(ii) By part~(i), we have 
\[ {\mathcal L}_{\xi}(f^p,g^p)\otimes{\mathcal L}_{\xi}(f,g)^{-1}\cong {\mathcal L}_{\xi}((f^p,g^p)-(f,g)).\]
Since the morphism $(f^p,g^p)-(f,g)$ factors through $L_p$, the rank one sheaf on the right-hand side is trivial. This proves the assertion. 
\qed}

We describe the relationship between ${\mathcal L}_{\xi}$ and Artin--Schreier sheaves. 
Let 
\[\xi\colon W_2(\F_p)\to\Ql^\times\]
be a character. Define a character $\psi\colon \F_p\to\Ql^\times$  as the composition  
\[\F_p\to W_2(\F_p)\xrightarrow{\xi}\Ql^\times,\]
where the first map is given by $x\mapsto (0,x)$. 
Let ${\mathcal L}_\psi$ be the rank one smooth sheaf on $\A^1_{\F_p}={\rm Spec}(\F_p[x])$ associated with the Artin--Schreier covering $y^p-y=x$ and the character $\psi$. 

For a morphism $f\colon X\to \A^1_{\F_p}$ of schemes, we set ${\mathcal L}_\psi(f):=f^\ast {\mathcal L}_\psi$. 
\begin{lemma}\label{xipsi}
Let $f\colon X\to \A^1_{\F_p}$ be a morphism of schemes. Then 
\[{\mathcal L}_\psi(f)\cong {\mathcal L}_\xi(0,f). \]
\end{lemma}
\proof{
The diagram 
\[\xymatrix{
\A^1_{\F_p}\ar[d]\ar[rr]^-{x \mapsto x^p-x}& &\A^1_{\F_p}\ar[d]\\
W_2\ar[rr]^-{L_p} & &W_2,}
\]
where the vertical arrows are given by $x\mapsto(0,x)$, is commutative. The assertion follows. 
\qed}

\begin{proposition}\label{rkLxi}
Let $k$ be an algebraically closed field containing $\F_p$. Assume that $\xi\in W_2(\F_p)^\vee$ has order $4$. Then, for any $t\in k$, we have 
\[
\dim H^1_c(\A^1_k,{\mathcal L}_\xi(x,tx))=1. 
\]
\end{proposition}
\proof{We first show that the dimension of $H^1_c(\A^1_k,{\mathcal L}_\xi(x,tx))$ is independent of the choice of $t$. Indeed, we have the equality $(x,tx)=(x+t,0)+(t,t^2)$. Therefore, by Lemma \ref{Lxi1}(i), we have 
\[{\mathcal L}_\xi(x,tx)\cong {\mathcal L}_\xi(x+t,0)\otimes{\mathcal L}_\xi(t,t^2). \]
Since ${\mathcal L}_\xi(t,t^2)$ is geometrically constant, we obtain an isomorphism 
\[
H^1_c(\A^1_k,{\mathcal L}_\xi(x,tx))\cong H^1_c(\A^1_k,{\mathcal L}_\xi(x+t,0)). 
\]
By the change of variable $x\mapsto x+t$, the right-hand side is isomorphic to $H^1_c(\A^1_k,{\mathcal L}_\xi(x,0))$. This proves the claim. 

We now compute the dimension for $t=0$. Consider the Lang morphism $L\colon \A^1_k\to\A^1_k$ given by $x\mapsto x^p-x=x^p+x$. 
By Lemma \ref{Lxi1} and the equality 
\[(x^p+x,0)=(x^p,0)-(x,0)+(0,x^2+x^{p+1}), \]
we have 
\[
L^\ast{\mathcal L}_\xi(x,0)={\mathcal L}_\xi(x^p+x,0)\cong
{\mathcal L}_\xi(0,x^2+x^{p+1}). 
\]
Let $\psi\colon \F_p\to\Ql^\times$ be the character defined by $\psi(x)=\xi(0,x)$. Since $\xi$ has order $4$, $\psi$ is nontrivial. By Lemma \ref{xipsi}, we have 
${\mathcal L}_\xi(0,x^2+x^{p+1})\cong {\mathcal L}_\psi(x^2+x^{p+1})$. Consequently, by the projection formula, we have 
\[L_\ast {\mathcal L}_\psi(x^2+x^{p+1})\cong {\mathcal L}_\xi(x,0)\otimes L_\ast\Ql. \]
We have an isomorphism $L_\ast\Ql\cong \bigoplus_{\psi'\in\F_p^\vee}{\mathcal L}_{\psi'}(x)$. Since $\psi$ is nontrivial, for each $\psi'\in\F_p^\vee$, there exists a unique $a\in \F_p$ such that $\psi'=\psi(a-)$. Therefore, 
\[L_\ast {\mathcal L}_\psi(x^2+x^{p+1})\cong \bigoplus_{a\in\F_p}{\mathcal L}_\xi(x,ax), \]
which induces an isomorphism 
\[H^1_c(\A_k^1,{\mathcal L}_\psi(x^2+x^{p+1}))\cong \bigoplus_{a\in\F_p}H^1_c(\A^1_k,{\mathcal L}_\xi(x,ax)).\]
By \cite[3.3]{TT}, the vector space on the left-hand side has dimension  $p$. Since we have already shown that
$\dim H^1_c(\A^1_k,{\mathcal L}_\xi(x,ax))$ are independent of $a$, each summand must have dimension $1$. 
\qed}

Fix a primitive $4$-th root of unity $\sqrt{-1}\in\Ql$. Since  the group $W_2(\F_2)$ is isomorphic to $\mathbb{Z}/4\mathbb{Z}$ with generator $(1,0)$, the choice $\sqrt{-1}$ defines a faithful additive character $\xi\colon W_2(\F_2)\to\Ql^\times$ by $\xi(1,0)=\sqrt{-1}$.

For a power $r=2^s$ of $2$, let 
\[
Q_r\colon\F_r\to\Ql^\times, \quad 
x \mapsto \xi \circ \Tr_{r/2}(x,0)=\xi\left(\sum_{i=0}^{s-1}(x^{2^i},0)\right). 
\]
The following result was also proved in \cite[Lemma 3.11]{ITT0} by a slightly different method. 
\begin{corollary}\label{HD}
For every power $r=2^s$ of $2$, we have an equality 
\[-\sum_{x\in \F_r}Q_r(x)=(-1-\sqrt{-1})^s. \]
\end{corollary}
\proof{Let $\F$ be an algebraic closure of $\F_r$. We apply Proposition \ref{rkLxi} with $p=2$ and $t=0$. By the Grothendieck--Lefschetz trace formula, we have 
\begin{equation}\label{LTF}
\sum_{x\in \F_r}Q_r(x)=\sum_{i=0}^2(-1)^i \Tr({\rm Fr}_r\mid H^i_c(\A^1_{\F},{\mathcal L}_\xi(x,0))). 
\end{equation}
The group 
$H^0_c(\A^1_{\F},{\mathcal L}_\xi(x,0))$ vanishes by the affine vanishing. Also, the group 
\[H^2_c(\A^1_{\F},{\mathcal L}_\xi(x,0))\cong H^0(\A^1_{\F},{\mathcal L}_\xi(x,0)^{-1}\otimes\Ql(-1))\]
vanishes since ${\mathcal L}_\xi(x,0)^{-1}\otimes\Ql(-1)$ is a nontrivial rank one sheaf. Therefore, \eqref{LTF} reduces to the equality 
\[-\sum_{x\in \F_r}Q_r(x)=\Tr({\rm Fr}_r \mid H^1_c(\A^1_{\F},{\mathcal L}_\xi(x,0))).\]
By Proposition \ref{rkLxi}, $H^1_c(\A^1_{\F},{\mathcal L}_\xi(x,0))$ has dimension $1$. It follows that 
\[\Tr({\rm Fr}_r\mid H^1_c(\A^1_{\F},{\mathcal L}_\xi(x,0)))=(\Tr({\rm Fr}_2 \mid H^1_c(\A^1_{\F},{\mathcal L}_\xi(x,0))))^s.\] The assertion follows. 
\qed}

\section{Computation of Frobenius Eigenvalues}\label{CFE}
Let $p$ be a power of $2$, and $q$ be a power of $p$. 

We first recall some basic results. Let \[R(x)=\sum_{i=0}^ea_ix^{p^i}\in \F_q[x],\] 
where $e\geq1$ and $a_e\neq0$. Consider the $\F_q$-curve 
\[
C_R\colon y^p+y=xR(x). 
\]
\begin{lemma}\label{g_C} 
The $\F_q$-curve $C_R$ is smooth and geometrically connected. Let $\overline{C}_R$ denote its smooth compactification. Then $\overline{C}_R$ has genus $p^e(p-1)/2$, and the complement $\overline{C}_R\setminus C_R$ consists of a single $\F_q$-rational point. 
\end{lemma}
\proof{Smoothness follows from the Jacobian criterion. The other properties follow readily from \cite[6.4.1]{St}. 
\qed}

We also recall the following result. Set $E:=R+R^\ast$, and define 
\[
V_q:=\ker E\cap \F_q. 
\]
\begin{lemma}\label{plin}
The map 
\begin{equation}\label{uR(u)}
    V_{q}\to\F_p,\quad u\mapsto \Tr_{q/p}(uR(u))
\end{equation}
is additive. 
\end{lemma}
\proof{We need to show that
\[\Tr_{q/p}((u+v)R(u+v))-\Tr_{q/p}(uR(u))-\Tr_{q/p}(vR(v))=0\]
for all $u,v\in V_{q}$. 
The left-hand side equals $\Tr_{q/p}(uR(v)+vR(u)).$ By Lemma \ref{b4}, this is equal to  
\[\Tr_{q/p}(uR(v)+uR^\ast(v))=\Tr_{q/p}(uE(v))=0. \]
\qed}

 \subsection{Setting and Definitions}\label{set}
 In this subsection, we describe the setting  used in what follows. Let $e\geq1$ be an integer. We consider an element 
 \begin{equation}\label{Fdef}
 F=\sum_{i=0}^eb_i\tau^i \in \F_q\{\tau^{\pm1}\}
 \end{equation}
 satisfying the following conditions: 
 \begin{enumerate}
 \item $b_0b_e\neq0$. 
 \item The function $F^\ast$ satisfies $F^\ast(1)=0$. 
 \end{enumerate}
 We also consider the following additional conditions on $F$: 
 \begin{enumerate}
 \item[(iii)] The kernel $W_F^\ast:=\ker (F^\ast\colon \F\to \F)$ is contained in $\F_q$. 
 \item[(iv)] The kernel $V_{F}:=\ker (F^\ast F\colon \F\to \F)$ is contained in $\F_q$. 
 \end{enumerate}
 Note that, by Lemma \ref{kfin}(i), condition (iii) is equivalent to the following:
 \begin{itemize}
 \item[(iii)$'$]  The kernel $W_F:=\ker (F\colon \F\to \F)$ is contained in $\F_q$. 
 \end{itemize}
The vector spaces $W_F^\ast$, $V_F$, and $W_F$ fit into the exact sequence 
\begin{equation}\label{WVV}
0\to W_F\to V_{F}\xrightarrow{F}W_{F}^\ast\to0. 
\end{equation}

 \begin{example}
 Let $W^\ast$ be an $\F_p$-vector subspace of $\F_q$ containing $\F_p$, and set $e:=\dim_{\F_p}W^\ast$. 
Write 
\[
\prod_{v\in W^\ast}(x-v)=\sum_{i=0}^ec_ix^{p^i},\quad c_i\in \F_q,
\]
and define $f:=\sum_{i=0}^ec_i\tau^i$. Then $F:=f^\ast\tau^{e}$ satisfies conditions 
\textup{(i)--(iii)}. 
 \end{example}
 
Assume that conditions (i)--(ii) hold. 
By Corollary \ref{dec2} and the description in \eqref{a}, $F^\ast F$ takes the form 
\[ \sum_{i=1}^e(a_i\tau^i+a_i^{p^{-i}}\tau^{-i}). \]
Explicitly, the coefficients are given by $a_i=\sum_{j=0}^{e-i}(b_jb_{j+i})^{p^{-j}}$. 

Set 
\[
R_F:=\sum_{i=1}^ea_i\tau^i,\qquad E_F:=F^\ast F. 
\]
Then 
\[
E_F=R_F+R_F^\ast.
\]

The following can be proven similarly to Lemma \ref{even}. 
\begin{lemma}\label{even2}
Assume further that \textup{(iii)} holds. Then $V_{F}\subset \F_{q^2}$. 
\end{lemma}
\proof{Write $q=p^n$. By the assumption and Lemma~\ref{basic}(iii), we can write 
\[\tau^n+1=GF^\ast\]
for some $G\in\F_q\{\tau^{\pm1}\}$. Hence, 
\[\tau^{2n}+1=\tau^n(\tau^n+1)(\tau^{-n}+1)=\tau^nGF^\ast FG^\ast. \]
Since the kernel of $\tau^{2n}+1$ is  $\F_{q^2}$, the assertion follows from Lemma \ref{b2}(ii). 
\qed}

We also introduce the following definitions. 
\begin{definition}\label{idf}
Define $\gamma_F\in\F_q$ by 
\[
\gamma_F:=\sum_{0\leq i<j\leq e}b_i^{p^{-i}}b_j^{p^{-j}}. 
\]
Moreover, define a map $\alpha_F\colon \F_q\to\F_q$ by 
\[\alpha_F(t):=\gamma_F+F^\ast(t)^2. 
\]
Finally, define 
\[
R_{F,t}(x):=R_F(x)+\alpha_F(t)x. 
\]
\end{definition}

The following results illustrate the use of the above definitions. 
For elements $\alpha,\beta\in W_2(\F_q[x])$, we write $\alpha\sim \beta$ if 
\[\alpha-\beta=(c^p,d^p)-(c,d)\]
for some $(c,d)\in W_2(\F_q[x])$. 
\begin{lemma}\label{F0}
Let $F$ be as in \eqref{Fdef}, satisfying conditions \textup{(i)} and \textup{(ii)}. 
Then  
\[(F(x),0)\sim (0,xR_{F,0}(x)). 
\]
\end{lemma}
\proof{We compute 
\begin{align*}
(F(x),0)&=\left(\sum_{i=0}^eb_ix^{p^i},0\right)\\
&=\sum_{i=0}^e(b_ix^{p^i},0)+\left(0,\sum_{0\leq i<j\leq e}b_ib_jx^{p^i}x^{p^j}\right)\\
&\sim\sum_{i=0}^e(b_i^{p^{-i}}x,0)+\left(0,\sum_{0\leq i<j\leq e}(b_ib_j)^{p^{-i}}x^{1+p^{j-i}}\right)\\
&=(F^\ast(1)x,\gamma_Fx^2)+(0,xR_F(x))=(F^\ast(1)x, x R_{F,0}(x)). 
\end{align*}
Here, the second and last equalities follow from Lemma \ref{ij} below. Since $F^\ast(1)=0$, the assertion follows. 
\qed}
\begin{lemma}\label{ij}
Let $R$ be an $\F_2$-algebra, and let $x_0,\dots,x_n\in R$. Then, in $W_2(R)$, we have  
\[
\sum_{i=0}^n(x_i,0)=\left(\sum_{i=0}^nx_i, \sum_{0\leq i<j\leq n}x_ix_j\right). 
\]
\end{lemma}
\proof{
The result follows by induction using the identity $(a,0)+(c,0)=(a+c,ac)$. 
\qed}

\begin{proposition}\label{Ima} 
Let $F$ be as in \eqref{Fdef}, satisfying conditions \textup{(i)} and \textup{(ii)}. 
Set 
\[W_{F,q}:=W_F\cap \F_q. \]
Then 
\[\alpha_F(\F_q)=\{a\in\F_q\mid \Tr_{q/p}\bigl(uR_{F}(u)+au^2\bigr)=0\text{ for all }u\in W_{F,q}\}. 
\]
\end{proposition}
\proof{Let $M$ denote the set on the right-hand side. Note that the condition 
\[\Tr_{q/p}\bigl(uR_{F}(u)+au^2\bigr)=0\]
is equivalent to 
\[
\Tr_{q/p}(uR_F(u))^{1/2}=\Tr_{q/p}(a^{1/2}u).
\]

Consider the $\F_p$-linear map 
\[\phi\colon \F_q\to {\rm Hom}_{\F_p}(W_{F,q},\F_p),\quad a\mapsto \Tr_{q/p}(a-)|_{W_{F,q}}.\]
By Lemma \ref{plin}, the map\footnote{Since $\F_p$ is a perfect field of characteristic $2$, every element of $\F_p$ admits a unique square root in $\F_p$. } 
\[W_{F,q}\to \F_p,\quad u\mapsto \Tr_{q/p}(uR_F(u))^{1/2}\]
is additive, and hence $\F_p$-linear. Then the set 
\[
M^{1/2}:=\{a^{1/2}\mid a\in M\} 
\]
is equal to the inverse image under $\phi$ of this linear map. Moreover, $\phi$
 is surjective, and $\ker\phi=F^\ast(\F_q)$ by Lemma \ref{kfin}(ii). Consequently, for any element $b\in M$, we have 
\[M=b+F^\ast(\F_q)^2. \]
It  remains to show that $\gamma_F\in M$. Let $u\in W_{F,q}$. Since $F(u)=0$, Lemma \ref{F0} implies that there exists  $c\in\F_q$ such that 
\[
uR_{F,0}(u)=c^p-c. 
\]
 Applying $\Tr_{q/p}$, we obtain 
 \[
 \Tr_{q/p}\bigl(uR_{F,0}(u)\bigr)=0. 
 \]
 This proves $\gamma_F\in M$, and the equality $M=\gamma_F+F^\ast(\F_q)^2=\alpha_F(\F_q)$ follows. 
\qed}

Let $\ell$ be an odd prime, and fix a primitive $4$-th root of unity $\sqrt{-1}$ in $\Ql$. Let 
\[\xi_2\colon W_2(\F_2)\to\Ql^\times\]
be the character defined by $\xi_2(1,0)=\sqrt{-1}$.

\begin{definition}\label{Qrx}
Let  $r$ be a power of $2$. 
\begin{enumerate}
\item Define the character $\xi_r\colon W_2(\F_r)\to\Ql^\times$  to be the composite 
\[
W_2(\F_r)\xrightarrow{\Tr_{r/2}} W_2(\F_2)\xrightarrow{\xi_2}\Ql^\times, 
\]
where $\Tr_{r/2}$ denotes the trace map. 
\item Define the map $Q_r\colon \F_r\to\Ql^\times$ by $Q_r(x):=\xi_r(x,0)$. 
\item Define the character $\psi_r\colon \F_r\to \Ql^\times$ by $\psi_r(x):=\xi_r(0,x)$. 
\item Define the map $B_{Q_r}\colon \F_r\times \F_r\to \Ql^\times$ by 
\[
B_{Q_r}(x,y):=Q_r(x+y)Q_r(x)^{-1}Q_r(y)^{-1}=\psi_r(xy).
\] 
\end{enumerate}
\end{definition}
The definition of $Q_r$ is consistent with that used in Corollary \ref{HD}. 
Note that $\psi_r$ can be expressed as 
 \begin{equation*}
 \psi_r(x)=(-1)^{\Tr_{r/2}(x)}, 
 \end{equation*}
where $\Tr_{r/2}$ denotes the trace map $\F_r\to\F_2$.

\subsection{Calculation of Frobenius Eigenvalues}
 Let $F$ be as in \eqref{Fdef}. We use the notation introduced in Subsection \ref{set}. For every element $t\in \F_q$, 
 define the $\F_q$-curve $D_{F,t}$ by 
\begin{equation}\label{Cadf}
D_{F,t}\colon y^p - y = x R_{F,t}(x), 
\end{equation}
where  $R_{F,t}(x):=R_F(x)+\alpha_F(t)x$. 

 The goal of this subsection is to describe the $\ell$-adic cohomology group 
 \[H^1_c(D_{F,t,\F},\Ql)\]
 as a representation of the geometric Frobenius ${\rm Fr}_q$. 
 

 Let $\F_p^\vee$ denote the character group ${\rm Hom}(\F_p,\Ql^\times)$. 
 For a nontrivial character $\psi\in \F_p^\vee$, let ${\mathcal L}_\psi$ be the Artin--Schreier sheaf on $\A^1_{\F_p}$ defined as in Section~\ref{langtorsor}. For a morphism $f\colon X\to \A^1_{\F_p}$ of schemes, we write $
 {\mathcal L}_\psi(f)$ for the pullback $f^\ast {\mathcal L}_\psi$. 

 Let 
 \[
 \phi\colon D_{F,t}\to \A^1_{\F_q}, \quad (x,y)\mapsto x. 
 \]
 \begin{lemma}\label{decAS}
 Assume that $F$ satisfies conditions \textup{(i)--(ii)}  stated after \eqref{Fdef}. Then the following hold: 
 \begin{enumerate}
 \item The pushforward $\phi_\ast\Ql$ is isomorphic to 
 \[
 \Ql\oplus\bigoplus_{\psi\in\F_p^\vee\setminus\{1\}}{\mathcal L}_\psi(xR_{F,t}(x)). \]
 \item There exists an ${\rm Fr}_q$-equivariant isomorphism 
 \[H^1_c(D_{F,t,\F},\Ql)\cong H^1_c(\A^1_{\F},{\mathcal L}_{\psi_p}(xR_{F,t}(x)))^{\oplus p-1}, \]
 where $\psi_p=(-1)^{\Tr_{p/2}(-)}$. 
 \end{enumerate}
 \end{lemma}
 \proof{(i) We have a cartesian diagram 
 \[\xymatrix{
 D_{F,t} 
 \ar[d]\ar[r]^-\phi&\A^1_{\F_q}\ar[d]^-{f}\\
 \A^1_{\F_q}\ar[r]^-{g}&\A^1_{\F_q},}\]
 where $f(x)=xR_{F,t}(x)$ and $g(y)=y^p-y$. By the proper base change theorem, we have $\phi_\ast\Ql\cong f^\ast g_\ast\Ql$. Since $g_\ast\Ql\cong \Ql\oplus\bigoplus_{\psi\in\F_p^\vee\setminus\{1\}}{\mathcal L}_\psi$, the assertion follows. 

 \medskip

 \noindent
(ii) Applying $H^1_c(\A^1_{\F},-)$ to (i), we obtain   
 \[H^1_c(D_{F,t,\F},\Ql)\cong H^1_c(\A^1_{\F},\Ql)\oplus\bigoplus_{\psi\in\F_p^\vee\setminus\{1\}}H^1_c(\A^1_{\F},{\mathcal L}_\psi(xR_{F,t}(x))). \]
The cohomology group $H^1_c(\A^1_{\F},\Ql)$ 
vanishes. Since $\psi_p$ is nontrivial, for every $\psi\in\F_p^\vee\setminus\{1\}$, there exists $b\in \F_p^\times$ such that $\psi=\psi_p(b-)$. Since $\F_p$ is perfect, $b$ admits a unique square root $b^{1/2}\in\F_p$. 
We can compute 
\[{\mathcal L}_\psi(xR_{F,t}(x))={\mathcal L}_{\psi_p}(bxR_{F,t}(x))=
{\mathcal L}_{\psi_p}(b^{1/2}xR_{F,t}(b^{1/2}x)). 
\]
Hence, the change of variable $x\mapsto b^{-1/2}x$ induces an isomorphism 
\[H^1_c(\mathbb{A}^1_{\F},{\mathcal L}_\psi(xR_{F,t}(x)))\cong H^1_c(\mathbb{A}^1_{\F},{\mathcal L}_{\psi_p}(xR_{F,t}(x))).\]
The assertion follows. 
 \qed}

Fix a primitive $4$-th root of unity $\sqrt{-1}\in\Ql$, and define $\xi_p$ as in Definition \ref{Qrx}. 
We compute $H^1_c(\mathbb{A}^1_{\F},{\mathcal L}_{\psi_p}(xR_{F,t}(x)))$ as follows. 
 \begin{proposition}\label{FEV}
 Assume that $F$ satisfies conditions \textup{(i)--(iii)}  stated after \eqref{Fdef}. Then the following statements hold: 
 \begin{enumerate}
 \item Consider the induced morphism $F\colon \A^1_{\F_q}\to \A^1_{\F_q}$. Then there is an isomorphism 
 \[
 F_\ast{\mathcal L}_{\psi_p}(xR_{F,t}(x))\cong \bigoplus_{v\in W_F^\ast}{\mathcal L}_{\xi_p}(x,(t+v)x). \]
 \item We have an isomorphism
 \[H^1_c(\A^1_{\F},{\mathcal L}_{\psi_p}(xR_{F,t}(x)))\cong \bigoplus_{v\in W_F^\ast}H^1_c(\A^1_{\F},{\mathcal L}_{\xi_p}(x,(t+v)x)). \]
 \end{enumerate}
  \end{proposition}
 \proof{
 Part (ii) follows from (i) by applying $H^1_c$. Thus it suffices to prove (i).  The pullback $F^\ast{\mathcal L}_{\xi_p}(x,tx)$ is isomorphic to ${\mathcal L}_{\xi_p}(F(x),tF(x))$. By Lemmas \ref{b4} and \ref{F0}, we have 
 \[(F(x),tF(x))\sim (0,xR_{F,0}(x)+F^\ast(t)x). 
 \]
 Therefore, by Lemmas \ref{Lxi1} and \ref{xipsi}, 
 \[F^\ast{\mathcal L}_{\xi_p}(x,tx)\cong {\mathcal L}_{\psi_p}(xR_{F}(x)+\gamma_Fx^2+F^\ast(t)x)\cong{\mathcal L}_{\psi_p}(xR_{F,t}(x)). \]
In the last isomorphism, we use the fact that ${\mathcal L}_{\psi_p}(f^2)\cong {\mathcal L}_{\psi_p}(f)$. 
By the projection formula, we obtain 
\[F_\ast{\mathcal L}_{\psi_p}(xR_{F,t}(x))\cong {\mathcal L}_{\xi_p}(x,tx)\otimes F_\ast\Ql. \]
 
 It remains to prove that
  \[F_\ast\Ql\cong \bigoplus_{v\in W_F^\ast}{\mathcal L}_{\psi_p}(vx). \]
  Since the pullback $F^\ast{\mathcal L}_{\psi_p}(vx)$ is isomorphic to 
  \[{\mathcal L}_{\psi_p}(vF(x))\cong {\mathcal L}_{\psi_p}(F^\ast(v)x)=\Ql,\]
  we obtain a nontrivial morphism ${\mathcal L}_{\psi_p}(vx)\to F_\ast\Ql$ by adjunction. Since ${\mathcal L}_{\psi_p}(vx)$ has rank $1$, and hence is irreducible, the morphism must be injective. Moreover, since the sheaves ${\mathcal L}_{\psi_p}(vx)$ are  pairwise non-isomorphic for distinct $v\in W_F^\ast$, the induced morphism 
  \[\bigoplus_{v\in W_F^\ast}{\mathcal L}_{\psi_p}(vx)\to F_\ast\Ql\]
  is also injective. By comparing ranks, we conclude that it is in fact an isomorphism. The assertion follows. 
 \qed}
 
 \begin{theorem}\label{FEV1}
 Assume that $F$ satisfies conditions \textup{(i)--(iii)} stated after \eqref{Fdef}. 
 
 Let $t\in\F_q$, and consider the curve $D_{F,t}$ defined by \eqref{Cadf}. 
  \begin{enumerate}
 \item 
 There exists an ${\rm Fr}_q$-equivariant isomorphism 
 \[H^1_c(D_{F,t,\F},\Ql)\cong \bigoplus_{v\in W_F^\ast}H^1_c(\A^1_{\F},{\mathcal L}_{\xi_p}(x,(t+v)x))^{\oplus p-1}. \]
 \item Each cohomology group $H^1_c(\A^1_{\F},{\mathcal L}_{\xi_p}(x,(t+v)x))$ has dimension $1$. 
 \item Let $\tau_v$ denote the ${\rm Fr}_q$-eigenvalue on  $H^1_c(\A^1_{\F},{\mathcal L}_{\xi_p}(x,(t+v)x))$.  Then 
 \[\tau_v=Q_q(t+v)^{-1}(-1-\sqrt{-1})^{[\F_q:\F_2]}.\]
 \item 
  Let $\overline{D}_{F,t}$ denote the smooth compactification of $D_{F,t}$. Then the $L$-polynomial 
\[
L(\overline{D}_{F,t}/\F_q,T):=\det(1-{\rm Fr}_qT \mid H^1(\overline{D}_{F,t,\F},\Ql))
\]
satisfies 
\[
L(\overline{D}_{F,t}/\F_q,T)=
\prod_{v\in W_F^\ast}(1-\tau_v T)^{p-1},
\]
where $\tau_v=Q_q(t+v)^{-1}(-1-\sqrt{-1})^{[\F_q:\F_2]}$. 
 \end{enumerate}
 \end{theorem}
 \proof{
 Part (i) follows from Lemma \ref{decAS}(ii) and Proposition \ref{FEV}(ii). Part (ii) is proved in Proposition \ref{rkLxi}. 
 
 We now prove (iii). Note that $H^i_c(\A^1_{\F},{\mathcal L}_{\xi_p}(x,(t+v)x))=0$ for $i\neq1$. Indeed, the vanishing for $i\neq1,2$ follows from the fact that  $\mathbb{A}^1_{\F}$ is affine of dimension $1$. For $i=2$, the Poincar\'e duality gives 
 \[H^2_c(\A^1_{\F},{\mathcal L}_{\xi_p}(x,(t+v)x))\cong H^0(\A^1_{\F},{\mathcal L}_{\xi_p}(x,(t+v)x)^\vee\otimes\Ql(1))^\vee. \]
 The  group $H^0(\A^1_{\F},{\mathcal L}_{\xi_p}(x,(t+v)x)^\vee\otimes\Ql(1))$ vanishes, since ${\mathcal L}_{\xi_p}(x,(t+v)x)^\vee$ is a nontrivial  rank one sheaf. 
 
 Therefore, by the Grothendieck--Lefschetz trace formula, we have
 \begin{equation*}
\tau_v=-\sum_{x\in\F_q}\xi_q(x,(t+v)x)=-\sum_{x\in\F_q}Q_q(x)\psi_q((t+v)x). 
\end{equation*}
Using  the identity $\psi_q((t+v)x)=B_{Q_q}(t+v,x)$, we obtain 
\begin{align*}
-\sum_{x\in\F_q}Q_q(x)\psi_q((t+v)x)
&=-\sum_{x\in\F_q}Q_q(x)B_{Q_q}(t+v,x)\\
&=-\sum_{x\in\F_q}Q_q(t+v+x)Q_q(t+v)^{-1}\\
&=Q_q(t+v)^{-1}\left(-\sum_{x\in\F_q}Q_q(x)\right). 
\end{align*}
The desired formula then follows from Corollary \ref{HD}. 

We prove (iv).  
By Lemma \ref{g_C}, the canonical map 
\[
H^1_c(D_{F,t,\F},\Ql)\to H^1(\overline{D}_{F,t,\F},\Ql)
\]
is an isomorphism. Then the assertion follows from (i), (ii), and (iii). 
 \qed}

The following identity relates the exponential sum associated with the van der Geer--van der Vlugt curve to an explicit character sum over $W_F^\ast$. 
 \begin{corollary}\label{gl}
Let the notation and assumptions be as in Theorem \ref{FEV1}. 
Then 
\[
\sum_{x \in \F_q} \psi_q(x R_{F,t}(x))
=-(-1-\sqrt{-1})^{[\F_q:\F_2]} \sum_{v \in W_F^\ast} 
Q_q(t+v)^{-1}. 
\]
\end{corollary}
\begin{proof}
The claim follows from the Grothendieck--Lefschetz trace formula, Proposition \ref{FEV}(ii), and 
Theorem \ref{FEV1}(iii). 
\end{proof}

\subsection{Supplement}
So far, we have started with a pair $(F,t)$, where  $F\in\F_q\{\tau^{\pm1}\}$ is as in \eqref{Fdef} and $t\in\F_q$, and studied the curve $D_{F,t}$. In this subsection, given a fixed van der Geer--van der Vlugt curve, we determine precisely  when it is isomorphic to $D_{F,t}$ for some pair $(F,t)$. 

We now formulate our results more precisely. Let $p$ be a power of $2$ and $q$ be a power of $p$. Fix an $\F_p$-linearized polynomial 
\[R(x)=\sum_{i=0}^ea_ix^{p^i}\in \F_q[x],\]
where $e\geq1$ and $a_e\neq0$. We determine when we can write $R=R_{F,t}$ for some $(F,t)$.

Set $E:=R+R^\ast$, and set 
\[
V:=\ker (E\colon \F\to \F), \qquad V_q:=V\cap \F_q. 
\]
Let $\omega$ denote the symplectic form on $V$ defined in Lemma~\ref{b5} applied to $E$.

The following is the main theorem of this subsection. 
\begin{theorem}\label{detR}
Let the notation be as above. Then the following conditions are equivalent: 
\begin{enumerate}
\item[$(1)$] There exists a pair $(F,t)$ such that 
\begin{itemize}
\item $F$ is  as in \eqref{Fdef} and satisfies conditions \textup{(i)--(iii)} thereafter, 
\item $t\in\F_q$, and 
\item $R=R_{F,t}$. 
\end{itemize}
\item[$(2)$] We have $V\subset \F_{q^2}$, and 
\[\Tr_{q^2/p}((u^q+u)R(u))=0 \] 
for all $u\in V$. 
\item[$(3)$] We have $V\subset \F_{q^2}$, and  
\[\Tr_{q/p}(uR(u))=0\]
for all $u\in V_{q}^\perp$. 
\item[$(4)$] There exists a maximal totally isotropic subspace $W\subset V$ with respect to $\omega$ such that $W\subset \F_q$ and 
\[
\Tr_{q/p}(uR(u))=0
\]
for all $u\in W$. 
\end{enumerate}
\end{theorem} 

\begin{remark}
    In \cite{ITT0}, the authors study the curve $C_R$ under a certain condition \cite[Assumption~2.4]{ITT0}. We claim that this condition is equivalent to the equivalent conditions in Theorem~\ref{detR}. 
    
    Indeed, let $H_R$ be the finite Heisenberg group defined in \cite[Definition~2.1]{ITT0}, whose underlying set is 
    \[
    H_R=\{(a,b)\in V\times\F\mid b^p+b=aR(a)\}. 
    \]
The condition in loc.~cit.~is the existence of a maximal abelian subgroup $A$ of $H_R$ such that $A\subset \F_q^2$. By \cite[Lemma~2.2]{ITT0} and Remark~\ref{omega in ITT0}, the set of such subgroups $A$ is in bijection with the set of 
 maximal totally isotropic subspaces $W\subset V$ such that $W\subset \F_q$ and, for every $u\in W$, the equation 
\[b^p+b=uR(u)\]
admits a solution in $\F_q$. The assignment $A\mapsto W$ is given by 
\[
W=\overline{A}, 
\]
where $\overline{A}$ denotes the image of $A$ under the projection $H_R\to V,\ (a,b)\mapsto a$. 

By Artin--Schreier theory, the equation 
\[b^p+b=uR(u)\]
admits a solution in $\F_q$ if and only if 
\[
\Tr_{q/p}(uR(u))=0. 
\]
Hence, \cite[Assumption~2.4]{ITT0} is equivalent to condition~$(4)$ in Theorem~\ref{detR}. 
\end{remark}

\begin{corollary}\label{VrFq}
If $R$ satisfies one of the following conditions, then the equivalent conditions in Theorem \ref{detR} hold. 
\begin{enumerate}
\item[$(a)$]  We have $V\subset\F_q$. 
\item[$(b)$]  We have $V\subset\F_{q^2}$, and for all $u\in V_q:=V\cap \F_q$, we have $\Tr_{q/p}(uR(u))=0$. 
\end{enumerate}
\end{corollary}
\proof{We show that condition $(2)$ in Theorem \ref{detR} holds in each case. 

If $V\subset \F_q$, then $u^q+u=0$ for all $u\in V$, and the condition is clearly satisfied. 

Now assume $(b)$. Let $u \in V$. Since $u\in \F_{q^2}$, 
we have $u^q+u \in \F_q$. 
Thus we obtain  
\[
\Tr_{q^2/p}\bigl((u^q+u)R(u)\bigr)=\Tr_{q/p}\bigl((u^q+u)R(u^q+u)\bigr). 
\]
Since $u \in V$, we have $E(u^q+u)=E(u)^q+E(u)=0$, and hence 
$u^q+u \in V$. Therefore, we have $u^q+u\in V_{q}$. Condition $(2)$ is thus satisfied if $\Tr_{q/p}(vR(v))=0$ for all $v\in V_{q}$. 
\qed}

\medskip

The remainder of this subsection is devoted to the proof of Theorem \ref{detR}. We begin with the easier implications. 
\begin{lemma}\label{i to iv to iii}
Consider the conditions in Theorem \ref{detR}. Then we have implications 
\[(1)\Longleftrightarrow(4)\Longrightarrow(3).\]
\end{lemma}
\begin{proof}
We first prove $(1)\Rightarrow(4)$. Let $(F,t)$ be a pair satisfying condition $(1)$. Set $W:=W_F=\ker F$. Then $W\subset \F_q$, and, by Proposition~\ref{b6} applied to the factorization $E=F^\ast F$, we have $W=W^\perp$. Thus, $W$ is maximal totally isotropic. 

It remains to show that 
\[
\Tr_{q/p}\bigl(uR(u)\bigr)=0
\]
    for all $u\in W$. This follows from Proposition~\ref{Ima}, since 
    \[R(x)=R_{F,t}(x)=R_F(x)+\alpha_F(t)x.  \]

We next prove the converse implication $(4)\Rightarrow(1)$. Let $W\subset V$ be a subspace as in $(4)$. Define 
\[f_W(x):=\prod_{v\in W}(x-v)\in \F_q[x].\]
By Corollary~\ref{dec2}, there exists $a\in \F_q^\times$ such that 
\[
E=f_W^\ast af_W. 
\]
Set $F:=a^{1/2}f_W$. Then $F$ clearly satisfies conditions (i) and (iii)$'$ stated after \eqref{Fdef}. Since (iii)$'$ is equivalent to (iii), the polynomial $F$ also satisfies (iii). 
Moreover, condition~\textup{(ii)} follows from the equality
\[
E=F^\ast F
\]
and Corollary~\ref{dec2}.

It remains to prove the existence of $t\in \F_q$ such that $
a_0=\alpha_F(t)$. This follows from Proposition~\ref{Ima} and the assumption on $W$.

\medskip 

\noindent
Finally, we prove $(4)\Rightarrow(3)$. Let $W\subset V$ be a subspace as in $(4)$. Then 
\[
W\subset V\cap \F_q=V_q, 
\]
and hence $V_q^\perp\subset W^\perp$. Since $W$ is maximal totally isotropic, we have $W^\perp=W$. Therefore,
\[
V_q^\perp\subset W,
\]
and consequently
\[
\Tr_{q/p}(uR(u))=0
\qquad
\text{for all }u\in V_q^\perp.
\]
It remains to show that $V\subset \F_{q^2}$. Let $u\in V$, and set $v:=u^q+u$. Since 
\[
E(v)=E(u)^q+E(u)=0, 
\]
we have $v\in V$. On the other hand, for any $w\in W$, we have 
\[
\omega(v,w)=\omega(u^q,w)+\omega(u,w)=0
\]
where the last equality follows from Lemma~\ref{b5}(iv) and the inclusion $W\subset \F_q$. Since $W=W^\perp$, it follows that $v\in W$. In particular, $v\in \F_q$. Therefore,
\[
u^{q^2}-u=v^q-v=0,
\]
which proves that $u\in \F_{q^2}$. 
This completes the proof. 
\end{proof}

Consequently, it remains to prove the implications 
\[
(2)\Longleftrightarrow(3)\Longrightarrow(4). 
\]
In particular, from now on, we may assume that 
\[V\subset\F_{q^2}. \]

We introduce some notation. Define the $\F_p$-linear map  
\[T\colon V\to V,\quad u\mapsto u^q.\]
This is well-defined since the coefficients of $E$ lie in $\F_q$. We also define $N:=T+{\rm id}_V$. 
Since $V\subset \F_{q^2}$, we have $N^2=0$. On the other hand, 
since $\ker N=V_q$, we have a short exact sequence 
\begin{equation}\label{VRqex}
0\to V_{q}\to V\xrightarrow{N}\mathop{\rm Im}N\to 0. 
\end{equation}
\begin{lemma}\label{uuu}
Let the notation and assumptions be as above. Then the following statements hold: 
\begin{enumerate}
\item[$(a)$] We have 
\[V_{q}^\perp=\mathop{\rm Im}N,\qquad\mathop{\rm Im}N\subset V_{q}.\]
\item[$(b)$] We have 
\[
(2)\Longleftrightarrow(3). 
\]
\end{enumerate}
\end{lemma}
\proof{$(a)$ Since $N^2=0$, the inclusion $\mathop{\rm Im}N\subset V_{q}$ follows from the exact sequence \eqref{VRqex}. Let $u\in V$ and $v\in V_{q}$. Then we have 
\[\omega(Nu,v)=\omega(u^q,v)+\omega(u,v)=0\]
by $v^q=v$ and Lemma \ref{b5}(iv). This shows that  $\mathop{\rm Im}N\subset V_{q}^\perp$. On the other hand, we have 
\[\dim V_{q}^\perp=\dim V-\dim V_{q}, \]
which equals $\dim\mathop{\rm Im}N$ by \eqref{VRqex}. The assertion follows. 

\medskip 

\noindent
$(b)$ Let $u \in V$. Since $\Ima N \subset V_q$ by (a), 
we have $u^q+u=Nu\in \F_q$. Thus, 
\[
\Tr_{q^2/p}\bigl((u^q+u)R(u)\bigr)=\Tr_{q/p}\bigl((u^q+u)\Tr_{q^2/q}(R(u))\bigr)=\Tr_{q/p}\bigl((Nu)R(Nu)\bigr). 
\]
Hence, condition $(2)$ is equivalent to the condition that $v \mapsto \Tr_{q/p}(vR(v))$ vanishes on $\mathop{\rm Im}N$. The assertion then follows since $\mathop{\rm Im}N=V_{q}^\perp$ by $(a)$. 
\qed}

\vspace{3mm}

\noindent\textit{(Proof of Theorem~\ref{detR})}

{By Lemmas~\ref{i to iv to iii} and~\ref{uuu}$(b)$, it remains to prove $(3)\Rightarrow(4)$. 

Assume that condition~$(3)$ holds. By Lemma~\ref{uuu}$(a)$, we have 
\[V_q^\perp\subset V_q.\]
Consider the map 
\[
\varphi \colon V_{q}\to\F_p, \quad u \mapsto \Tr_{q/p}(uR(u))^{1/2}.
\]
By Lemma \ref{plin}, this map is additive, and hence $\F_p$-linear. By condition $(3)$, we have 
\[
V_{q}^\perp\subset \ker \varphi. 
\]
Choose a subspace $X\subset V_{q}$ such that $V_{q}=X\oplus V_{q}^\perp$. Let $f=\dim V_{q}^\perp$. Then, by \eqref{VRqex} and Lemma \ref{uuu}$(a)$, we have 
\[
\dim X=\dim V_q-\dim V_q^{\perp}=\dim V-2 \dim V_q^{\perp}=
2(e-f). 
\]

Applying Lemma \ref{tis} below to the triple $(X,\omega|_X,\varphi)$, we obtain a totally isotropic subspace $W'\subset X$ of dimension $e-f$ such that  $\varphi$ vanishes on $W'$. Then the space $W=W'+V_{q}^\perp$ satisfies all the properties required in $(4)$. 
\qed}

\begin{lemma}\label{tis}
Let $k$ be a field, and let $X$ be a $k$-vector space of even dimension $2r$. Let $\omega$ be a symplectic pairing on $X$, and let $\varphi$ be a $k$-linear map $X\to k$. Then there exists a totally isotropic subspace $W'\subset X$ of dimension $r$  such that $\varphi|_{W'}=0$. 
\end{lemma}
\proof{First, we consider the case where $\varphi=0$. Since $\dim X=2r$, any maximal totally isotropic subspace has dimension $\geq r$. The assertion follows in this case. 

Next, assume $\varphi\neq0$. Let $K=\ker \varphi$; then  $\dim K=2r-1$. Since $K$ is odd-dimensional, the restriction $\omega|_{K\times K}$ is degenerate. Let $L\subset K$ be the radical, defined as  the subspace consisting of elements $u\in K$ such that $\omega(u,v)=0$ for all $v\in K$. Since $\omega|_{K\times K}$ is degenerate, $L\neq0$. Furthermore, we have an induced pairing 
\[\overline{\omega}\colon (K/L) \times (K/L)\to k, \]
which is nondegenerate and symplectic. Choose a maximal totally isotropic subspace $\overline{A}\subset K/L$, and let $A$ be its inverse image in $K$. Then we have 
\[\dim A=\dim L+\frac{\dim K-\dim L}{2}=\frac{2r-1+\dim L}{2}. \]
Since $L\neq0$, it follows that $\dim A\geq r$. By construction, $A$ is totally isotropic and satisfies $\varphi|_A=0$. Then we can choose $W'$ to be an $r$-dimensional subspace of $A$. 
\qed}

\section{Maximal and minimal twists of van der Geer--van der Vlugt Curves}\label{max2}
Let $p$ be a power of $2$, and $q$ be a power of $p$. 
\subsection{Review of Results from \cite{GV}}
\label{notation_twist}
Fix a polynomial 
\[R(x)=\sum_{i=1}^ea_ix^{p^i}\in \F_q[x],\]
where $e\geq1$ and $ a_e\neq0$.  For each $a\in \F_q$, set $R_a(x):=R(x)+ax$, and let $C_a$ denote the affine $\F_q$-curve defined by 
\begin{equation}\label{GVCa}
 y^p - y = x R_a(x). 
 \end{equation}

In this section, we regard the curves $C_a$ as a family  parameterized by $a$. Such families were studied extensively by van der Geer--van der Vlugt in \cite{GV}. In the remainder of this subsection, we review their results.

Let $E:=R+R^\ast$, and define  
\[V:=\ker(E\colon \F\to \F),\qquad V_{q}:=V\cap \F_q. \]
Note that $E=R_a+R_a^\ast$ for every $a$, since $a+a=0$. 

The following theorem can be proved in essentially the same way as \cite[Section 5]{GV}. 
\begin{theorem}\label{Npm}
Let $a\in \F_q$. Then the following statements hold: 
\begin{enumerate}
\item If the additive map \eqref{uR(u)}
is nonzero, then 
\[|C_a(\F_q)|=q. \]
\item If the additive map 
\eqref{uR(u)} 
is zero, then  
\[|C_a(\F_q)|=q\pm(p-1)\sqrt{qp^d}, \]
where $d=\dim_{\F_p}V_{q}$. 
\end{enumerate}
\end{theorem}
\proof{
Let $\psi_p$ be the character $\F_p\to\C^\times,x\mapsto (-1)^{\Tr_{p/2}(x)}$. Then any additive character $\F_p\to\C^\times$ can be written uniquely as $\psi_{p,b}:=\psi_p(b-)$ for some $b\in\F_p$.
The sequence 
\[0\to\F_p\to\F\xrightarrow{x^p-x}\F\to0\]
is exact. Taking  Galois cohomology $H^i({\rm Gal}(\F/\F_q),-)$, we obtain an exact sequence 
\[\F_q\xrightarrow{x^p-x}\F_q\xrightarrow{\Tr_{q/p}}\F_p\to0.  \]
Therefore, 
for $z\in \F_q$, the equation $y^p-y=z$ has a solution in $\F_q$ if and only if $\Tr_{q/p}(z)=0$. This is equivalent to  $\psi_{p,b}(\Tr_{q/p}(z))=1$ for all  $b\in\F_p$. Furthermore, if $y^p-y=z$ has a solution in $\F_q$, then it has exactly $p$ solutions. 

Thus we have 
\begin{align*}
|C_a(\F_q)|&=\sum_{b\in\F_p}\sum_{x\in\F_q}\psi_{p,b}\bigl(\Tr_{q/p}(xR_a(x))\bigr)\\
&=q+\sum_{b\neq0}\sum_{x\in\F_q}\psi_{p,b}\bigl(\Tr_{q/p}(xR_a(x))\bigr). 
\end{align*}
Note that  
\[\psi_{p,b}\bigl(\Tr_{q/p}(xR_a(x))\bigr)=\psi_p\bigl(\Tr_{q/p}(bxR_a(x))\bigr)=\psi_p\bigl(\Tr_{q/p}((b^{1/2}x)R_a(b^{1/2}x))\bigr). \]
Here $b^{1/2}$ exists in $\F_p$ since $\F_p$ is perfect, and therefore the last equality holds. 
Consequently, the sum $\sum_{x\in\F_q}$ for $\psi_{p,b}$ with $b\neq0$ is equal to that for $\psi_p$. Hence, we obtain  
\[|C_a(\F_q)|=q+(p-1)\sum_{x\in\F_q}(-1)^{\Tr_{q/2}(xR_a(x))}. \]
Hence, we are reduced to determining the sum of $(-1)^{\Tr_{q/2}(xR_a(x))}$. This is done in \cite[(5.2)]{GV} and \cite[(5.3)]{GV}. 
\qed}


Now we recall the following terminology on curves. 
\begin{definition}\label{mmx}
Let $\overline{C}$ be a geometrically connected smooth projective curve of genus $g$ over a finite field $k$. Let $L/k$ be  a finite field extension with $|L|=r$. Then  the Hasse--Weil bound gives 
\[1+r-2g\sqrt{r}\leq|\overline{C}(L)|\leq 1+r+2g\sqrt{r}.\]

We say that $\overline{C}$ is  \emph{$L$-maximal} (resp.\ \emph{$L$-minimal}) if the upper bound (resp.\ the lower bound) is achieved. We say that $\overline{C}$ is \emph{$L$-extremal} if it is either $L$-maximal or $L$-minimal. \end{definition}
Let $\overline{C}_a$ denote the smooth compactification of $C_a$. We shall simply say that $C_a$ is $\F_q$-maximal, $\F_q$-minimal, or $\F_q$-extremal if $\overline{C}_a$ is $\F_q$-maximal, $\F_q$-minimal, or $\F_q$-extremal, respectively.  
\begin{lemma}\label{mmeCa}
The curve $C_a$ is $\F_q$-extremal if and only if $H^1_c(C_{a,\F},\Ql)$ has a single ${\rm Fr}_q$-eigenvalue. Moreover, the curve $C_a$ is $\F_q$-maximal (resp.\ $\F_q$-minimal) if and only if $-\sqrt{q}$ (resp.\ $\sqrt{q}$) is the unique ${\rm Fr}_q$-eigenvalue of $H^1_c(C_{a,\F},\Ql)$. 
\end{lemma}
\proof{By Lemma \ref{g_C}, the canonical map $H^1_c(C_{a,\F},\Ql)\to H^1(\overline{C}_{a,\F},\Ql)$ is an isomorphism. Then the assertion follows from the Grothendieck--Lefschetz trace formula and the Riemann hypothesis for curves over finite fields.
 \qed}

Theorem \ref{Npm} has the following consequence. 
\begin{corollary}\label{exN}
 If there exists $a\in \F_q$ such that $C_a$ is $\F_q$-extremal, then 
 \[
 V_q=V. 
 \]
  In other words, we have $V\subset \F_q$. 
\end{corollary}
 \proof{
 By Lemma \ref{g_C}, the genus of $\overline{C}_a$ is equal to $p^e(p-1)/2$. Furthermore, $\overline{C}_a\setminus C_a$ consists of a single $\F_q$-rational point. Therefore, by Theorem \ref{Npm}, the extremality of $C_a$ implies 
 \[p^e(p-1)=(p-1)\sqrt{p^d}.\]
 Hence, we have $d=2e$, which is equivalent to $V_{q}=V$. The assertion follows. 
 \qed}

From now on, we assume that $V\subset \F_q$. 
Define 
\begin{gather}\label{Ndf}
\begin{aligned}
T_{R,0} &:= \{a\in\F_q\mid |C_a(\F_q)|=q \},  \\
T_{R,\max} &:= \{a\in\F_q\mid |C_a(\F_q)| = q + (p-1)p^e\sqrt{q}\},  \\
T_{R,\min} &:= \{a\in\F_q\mid |C_a(\F_q)| = q - (p-1)p^e\sqrt{q}\}. 
\end{aligned}
\end{gather}
By Theorem~\ref{Npm}, we have the decomposition 
\[\F_q=T_{R,0}\sqcup T_{R,\max}\sqcup T_{R,\min}\]
and the equalities  
\begin{gather}\label{N0pm}
\begin{aligned}
T_{R,0}&=\{a \in \F_q \mid \exists u\in V,\,\Tr_{q/p}(uR_a(u))\neq0\}, \\
T_{R,\max} \cup T_{R,\min}&=\{a \in \F_q \mid \forall u\in V,\,\Tr_{q/p}(uR_a(u))=0\}. 
\end{aligned}
\end{gather}

Since the genus of $\overline{C}_a$ is equal to $p^e(p-1)/2$, the set $T_{R,\max}$ (resp.\ $T_{R,\min}$) agrees with the set of $a\in \F_q$ such that $\overline{C}_a$ is $\F_q$-maximal (resp.\ $\F_q$-minimal). 

\medskip 

We record the following for later use. 
 \begin{lemma}\label{even}
 Let 
 \[R(x)=\sum_{i=1}^e a_i x^{p^i} \in 
 \F_q[x],\] 
 where $e\geq1$ and $a_e\neq0$. Set 
 \[
 E:=R+R^\ast,\qquad V:=\ker(E\colon \F\to \F).\]
 If $V\subset \F_q$, then $[\F_q:\F_p]$ is even. 
 \end{lemma}
 \proof{Let $\omega$ be the symplectic pairing on $V$ defined in Lemma \ref{b5}(iii) applied to $F=E$. By choosing a maximal totally isotropic subspace with respect to $\omega$, we obtain a decomposition $E=F^\ast F$ as in Corollary \ref{dec2}. Furthermore, since the constant term of $E$ is zero, the same corollary implies that $F^\ast(1)=0$. Therefore, we can write $F^\ast=G(\tau+1)$ for some $G\in \F_q\{\tau^{\pm1}\}$. Thus we have 
 $E=G(\tau+\tau^{-1})G^\ast$. Since  $V\subset \F_q$, Lemma \ref{b2}(iii) implies that $\F_{p^2}\subset\F_q$, as desired. 
 \qed}

 \subsection{Explicit Descriptions of $T_{R,\max}$ and $T_{R,\min}$}\label{relNpm}
 
 In this subsection, we describe $T_{R,\max}$ and $ T_{R,\min}$ by using our results in Section~\ref{CFE}. 
 
Let 
\[R(x)=\sum_{i=1}^e a_i x^{p^i} \in 
 \F_q[x],\]
 where $e\geq1$ and $a_e\neq0$. Assume that the kernel of $E:=R+R^\ast$ is contained in $\F_q$. Under this assumption, there always exists $F$ satisfying $R=R_F$. 
 \begin{lemma}\label{F always exists}
There exists an element $F$ as in \eqref{Fdef} satisfying the following properties: 
\begin{enumerate}
    \item The element $F$ satisfies conditions \textup{(i)--(iv)} stated after \eqref{Fdef}. 
    \item We have $R=R_F$. 
\end{enumerate}
 \end{lemma}
\begin{proof}
By assumption, the condition $(a)$ of Corollary~\ref{VrFq}
is satisfied. The claim thus follows since condition~$(1)$ of Theorem~\ref{detR} holds. 
\end{proof}
 
From now on, we choose and fix $F$ as in the above lemma. 
We use the notation introduced in Subsection~\ref{set}. 
 Under the assumption $\ker E\subset \F_q$, the sets $T_{R,\max}$ and $T_{R,\min}$ in \eqref{Ndf} correspond to the $\F_q$-maximal and $\F_q$-minimal curves in the family $\{C_a\}_{a\in\F_q}$. 
 
 Let $Q_q$ be as in Definition \ref{Qrx}. 

\begin{lemma}\label{ghom}
The restriction of $Q_q$ to $W_F^\ast$ 
\[Q_q|_{W_F^\ast}\colon W_F^\ast\to\Ql^\times\]
is a group homomorphism. 
\end{lemma}
\proof{
It suffices to show that $B_{Q_q}(x,y)=1$ for all $x,y\in W_F^\ast$. Since  
\[B_{Q_q}(x,y)=\psi_q(xy)=(-1)^{\Tr_{q/2}(xy)},\]
it suffices to prove that $\Tr_{q/2}(xy)=0$. 
From \eqref{WVV}, we can write $x=F(z)$ for some $z\in V_{F}$. Then, by Lemma \ref{b4}, we have 
\[xy=F(z) y\sim zF^\ast(y)=0.\]
The assertion follows. 
\qed}

We have the following vanishing result for the exponential sums.
\begin{corollary}
Assume that $Q_q|_{W^\ast_F}$ is nontrivial. Then, for $t \in W_F^\ast$, we have
\[
\sum_{x \in \F_q} \psi_q(x R_{F,t}(x))=0. 
\]
\end{corollary}
\begin{proof}
The assertion follows from Corollary \ref{gl} and Lemma \ref{ghom}. 
\end{proof}
For a finite abelian group $G$, we set $G^\vee:=\mathop{\rm Hom}(G,\Ql^\times)$. Consider the composite map 
\[\F_q\xrightarrow{x\mapsto\psi_q(x-)}
\F_q^\vee\to (W_F^\ast)^\vee,\]
where the second arrow is induced by the inclusion $W_F^\ast\hookrightarrow \F_q$. This map is surjective. 
\begin{lemma}\label{kerW_F}
The kernel of the surjection $\F_q\to (W^\ast_F)^\vee$ is equal to $F(\F_q)$. 
\end{lemma}
\proof{
The proof is analogous to that of Lemma \ref{kfin}(ii). 
\qed}

We define a subset $S_{F}$ of $\F_q$ by  
\begin{equation}\label{Spmdf}
 S_{F}:=\{t\in\F_q\mid \forall v\in W_F^\ast,\, Q_q(v)=\psi_q(tv)\}. 
\end{equation}
We set $S_0:=\F_q\setminus S_F$. 
Furthermore, we define 
\begin{gather}\label{s+-}
\begin{aligned}
&S_{F,-}:=\{t\in S_F\mid Q_q(t)=-(\sqrt{-1})^{[\F_q:\F_2]/2}\},\\
&S_{F,+}:=\{t\in S_F\mid Q_q(t)=(\sqrt{-1})^{[\F_q:\F_2]/2}\}. 
\end{aligned}
\end{gather}
Note that $[\F_q:\F_2]$ is even by Lemma \ref{even}. 

We also consider the map $\alpha_F$ defined in Definition \ref{idf}. 
\begin{lemma}\label{ptor}
The map $\alpha_F$ is a pseudo-torsor over $W_F^\ast$. In other words, every fiber of $\alpha_F$ is either empty or of the form $a+W_F^\ast$. 
\end{lemma}
The following is the main theorem of this subsection. 
\begin{theorem}\label{mainthm}
The following statements hold: 
\begin{enumerate}
\item We have a decomposition 
\[\F_q
=S_0\sqcup S_{F,-}\sqcup S_{F,+}. \]
\item We have 
\[\alpha_F(S_0)\subset T_{R,0}, \quad \alpha_F(S_{F,-})=T_{R,\max}, \quad \alpha_F(S_{F,+})=T_{R,\min}.\]
\end{enumerate} 
\end{theorem}
\proof{Let $t\in\F_q$. By Theorem \ref{FEV1}, the set of the ${\rm Fr}_q$-eigenvalues on $H^1_c(D_{F,t,\F},\Ql)$ is equal to 
\[\{Q_q(t+v)^{-1}(\sqrt{-1})^{[\F_q:\F_2]/2}2^{[\F_q:\F_2]/2}\mid v\in W_F^\ast\}.\]
This set consists of a single value if and only if $Q_q(t)=Q_q(t+v)$ for any $v\in W_F^\ast$. We compute 
\begin{align*}
Q_q(t+v)Q_q(t)^{-1}=B_{Q_q}(t,v)Q_q(v)=\psi_q(tv)Q_q(v). 
\end{align*}
This is identically equal to $1$ if and only if $t\in S_F$. By Lemma \ref{mmeCa}, it follows that 
\[\alpha_F(S_0)\subset T_{R,0}, \quad \alpha_F(S_{F})\subset T_{R,\max}\cup T_{R,\min}.\]

Now assume $t\in S_F$. In this case, the curve $D_{F,t}$ is $\F_q$-extremal. It is $\F_q$-maximal (resp.\ $\F_q$-minimal) if and only if 
\[
Q_q(t)^{-1}(\sqrt{-1})^{[\F_q:\F_2]/2}=-1\qquad\text{(resp.\ }=1\text{)}. 
\]
This shows the $\alpha_F(S_{F,-})\subset T_{R,\max}$ and $\alpha_F(S_{F,+})\subset T_{R,\min}$. It also proves~(i).

It remains to prove
\[
  \alpha_F(S_{F,-}) = T_{R,\max}, \quad
  \alpha_F(S_{F,+}) = T_{R,\min}.
\]
Since we already know that $\F_q=T_{R,0}\sqcup T_{R,\max}\sqcup T_{R,\min}$ and
\[
\alpha_F(S_0)\subset T_{R,0},\quad 
  \alpha_F(S_{F,-}) \subset T_{R,\max}, \quad
  \alpha_F(S_{F,+}) \subset T_{R,\min}, 
\]
it suffices to show $T_{R,\max} \cup T_{R,\min}\subset \alpha_F(\F_q)$. 
This follows from Proposition \ref{Ima} together with  \eqref{N0pm}. 
\qed
}

\section{Examples of Extremal Curves and Periods}
\label{constrmaxcurve}
In this section, we provide various constructions of extremal van der Geer--van der Vlugt curves, as consequences of the results established above. Let $p$ be a power of $2$.

\subsection{Maximal $\F_{q^2}$-curves}
Let $q$ be a power of $p$. Assume that $[\F_q:\F_p]$ is even. 
Let $F$ be an $\F_p$-linearized polynomial as in~\eqref{Fdef} satisfying conditions \textup{(i)--(iv)} thereafter. The following gives examples of $\F_q$-curves that are maximal over $\F_{q^2}$. 
\begin{theorem}\label{3}
Let $t$ be an element of $\F_q$ satisfying $\Tr_{q/2}(t)=([\F_q:\F_2]/2)+1$. Then the curve $D_{F,t}$ is $\F_{q^2}$-maximal. 
\end{theorem}
\proof{By the assumption $\Tr_{q/2}(t)=([\F_q:\F_2]/2)+1$, the quantity $Q_q(t)^{-1}(\sqrt{-1})^{[\F_q:\F_2]/2}$ is a primitive $4$-th root of unity. On the other hand, for any $v\in W_F^\ast$, we have 
\[
Q_q(t+v)Q_q(t)^{-1}=B_{Q_q}(v,t)Q_q(v). 
\]
Note that $B_{Q_q}(v,t) \in \{\pm 1\}$. 
Since $W_F^\ast$ is an $\F_2$-vector space, 
we also have $Q_q(v) \in \{\pm 1\}$ by Lemma \ref{ghom}. 

Therefore, by Theorem \ref{FEV1}(iii), the ${\rm Fr}_q$-eigenvalues are of the form $\pm(\sqrt{-1})2^{[\F_q:\F_2]/2}$. It follows that all ${\rm Fr}_{q^2}$-eigenvalues are equal to $-2^{[\F_q:\F_2]}$. This shows that $D_{F,t}$ is $\F_{q^2}$-maximal. 
\qed}

\subsection{Linear-Algebraic Construction of $\F_q$-extremal Curves}
Let $q$ be a power of $p$. Assume that $[\F_q:\F_p]$ is even. 

As a consequence of Theorem~\ref{mainthm}, we give a recipe for constructing $\F_q$-extremal van der Geer--van der Vlugt curves from linear-algebraic data. 
Consider a pair $(W,t)$, where 
\begin{itemize}
    \item $W$ is an $\F_p$-linear subspace of $\F_q$ such that $1\in W$. 
    \item $t$ is an element of $\F_q$ such that 
    \[
    Q_q(v)=\psi_q(tv)\qquad\text{for all } v\in W. 
    \]
\end{itemize}
To such a pair $(W,t)$, we define 
\[
f_W(x):=\prod_{v\in W}(x-v),\qquad F_W:=(\tau^{-e}f_W)^\ast, 
\]
where $e:=\dim_{\F_p}W$. One can check that $F_W$ satisfies conditions \textup{(i)--(iii)} stated after \eqref{Fdef}. 
Let $D_{F_W,t}$ be the $\F_q$-curve defined for the pair $(F_W,t)$. 
\begin{theorem}\label{recipe for ex curve}
The following statements hold: 
\begin{enumerate}
    \item The $\F_q$-curve $D_{F_W,t}$ is $\F_q$-extremal. It is $\F_q$-maximal (resp.\ $\F_q$-minimal) if and only if 
    \[
    Q_q(t)=-(\sqrt{-1})^{[\F_q:\F_2]/2}\qquad\text{(resp.\ }Q_q(t)=(\sqrt{-1})^{[\F_q:\F_2]/2}\text{)}. 
    \]
    \item Every $\F_q$-extremal van der Geer--van der Vlugt curve is obtained in this way. More precisely, let 
\[R(x)=\sum_{i=0}^e a_i x^{p^i} \in 
 \F_q[x],\]
 where $e\geq1$ and $a_e\neq0$, and assume that $C_R$ is $\F_q$-extremal. Then there exists a pair $(W,t)$ as above and $a\in\F_q^\times$ such that $(x,y)\mapsto (ax,y)$ induces an isomorphism 
 \[
 D_{F_W,t}\xrightarrow{\cong}C_R. 
 \]
\end{enumerate}
\end{theorem}
\begin{proof}
(i)  Since $\psi_q(tv)$ is a sign, we have $Q_q(v)=Q_q(v)^{-1}$ for all $v\in W$. Thus, the equality $Q_q(v)=\psi_q(tv)$ can be rewritten as 
\[
Q_q(v)^{-1}=\psi_q(tv)=Q_q(t+v)Q_q(v)^{-1}Q_q(t)^{-1}, 
\]
and hence $Q_q(t+v)=Q_q(t)$ is independent of $v\in W$. The assertion follows from Theorem~\ref{FEV1}. 

\medskip 

\noindent
(ii) Let $R(x)$ be an $\F_p$-linearized polynomial as in the statement. By Corollary~\ref{exN}, we have $\ker(R+R^\ast)\subset \F_q$. Hence, by Lemma~\ref{F always exists}, 
    we can write 
    \[
    R(x)-a_0x=R_{F}(x)
    \]
    for some $F$ satisfying conditions \textup{(i)--(iv)} stated after~\eqref{Fdef}. Moreover, by Theorem~\ref{mainthm}(ii), there exists $t\in S_{F}$ such that $a_0=\alpha_{F}(t)$. 

    Set $W:=\ker F^\ast$. Then, since $\ker F^\ast=\ker F_W^\ast$,  there exists $a\in\F_q^\times$ such that 
    \[
    F^\ast=aF_W^\ast. 
    \]
    Hence, 
    \[
    R+R^\ast=F^\ast F=aF_W^\ast F_W a=a(R_{F_W}+R_{F_W}^\ast)a. 
    \]
    Comparing the coefficients in positive degree, we obtain 
    \[
    R(x)=aR_{F_W}(ax)+\alpha_F(t)x. 
    \]
    Moreover, since $F(x)=F_W(ax)$, we have $\alpha_F(t)=a^2\alpha_{F_W}(t)$. Thus, 
    \[
    R(x)=aR_{F_W,t}(ax), 
    \]
    and the claim follows. 
\end{proof}

\subsection{More Explicit Descriptions of $S_{F,-}$ and $T_{R,\max}$}\label{MEDST}
Let $q$ be a power of $p$. 
Let $F$ be as in \eqref{Fdef}, and assume that conditions (i)--(iv) thereafter are satisfied. 
Then, by Lemma \ref{even}, the degree $[\F_q:\F_p]$ is even. 
We use the notation from Subsections \ref{set} and \ref{relNpm}.

Recall that 
\[
R_{F,t}(x)=R_F(x)+\alpha_F(t)x. 
\]
We set $R:=R_F$. 
\begin{theorem}\label{easycase}
We further assume that $[\F_q:\F_p]$ is divisible by $4$, and that $W_F^\ast\subset \F_{q^{1/4}}$. Then the following hold: 
\begin{enumerate}
\item We have 
\[S_F=F(\F_q),\qquad T_{R,\max} \cup T_{R,\min}=\gamma_F+E_F(\F_q)^2. 
\]
\item We have 
\begin{align*}
S_{F,-}&=\{\,F(u)\mid u\in \F_q,\ \Tr_{q/2}(uR_{F,0}(u))=([\F_q:\F_2]/4)+1\,\}, \\
T_{R,\max}&=\{\,\gamma_F+E_F(u)^2\mid u\in \F_q,\ \Tr_{q/2}(uR_{F,0}(u))=([\F_q:\F_2]/4)+1\, \}. 
\end{align*}
\end{enumerate}
\end{theorem}
\begin{proof}
We prove the assertions for $S_F$ and $S_{F,-}$; those for $T_{R,\max} \cup T_{R,\min}$ and $T_{R,\max}$ then follow from Theorem \ref{mainthm} and $E_F=F^\ast F$. 

(i) Since $W_F^\ast\subset \F_{q^{1/4}}$ and 
$W_2(\F_q)$ is annihilated by $4$, we have $Q_q(v)=1$ for all $v\in W_F^\ast$. Therefore, by \eqref{Spmdf}, we have 
$S_F=\{t \in \F_q \mid \forall v \in W_F^\ast,\ \psi_q(tv)=1\}$. 
Thus the equality $S_F=F(\F_q)$ follows from Lemma \ref{kerW_F}. 

(ii) We need to compute $Q_q(F(u))$. By Lemma \ref{F0}, we have 
\[
Q_q(F(u))=\psi_q(uR_{F,0}(u))=(-1)^{\Tr_{q/2}(uR_{F,0}(u))}. 
\]
 The claim then follows from the definition of $S_{F,-}$ in \eqref{s+-}. 
\end{proof}
\begin{corollary}
    Under the assumptions in Theorem \ref{easycase}, the curve $D_{F,0}$ is $\F_q$-extremal. It is $\F_q$-maximal if and only if $[\F_q:\F_2]/4$ is odd. 
\end{corollary}
 
In the following theorem, we consider a more general setting. 
\begin{theorem}\label{exmax}
Let $q_1$ be a power of $p$ such that $q_1^2\mid q$. Assume that $F$ satisfies conditions \textup{(i)--(iv)} stated after \eqref{Fdef}. We further assume that $W_F^\ast\subset \F_{q_1}$. Write  $q=q_1^{2n}$. Then the following hold: 
\begin{enumerate}
\item There exists an element $t_0\in\F_{q_1^2}$ such that $t_0^{q_1}+t_0=1$. 
\item For $t_0$ as in \textup{(i)}, we have 
\[
S_F=t_0+F(\F_q), \qquad 
T_{R,\max} \cup T_{R,\min}=\alpha_F(t_0)+E_F(\F_q)^2.\]

\item For $t_0$ as in \textup{(i)}, we have 
\begin{align*}
S_{F,-} &= \{\,t_0+F(u)\mid u\in \F_q,\ \Tr_{q/2}(uR_{F,t_0}(u))=n+1\,\},\\
T_{R,\max} &= \{\,\alpha_F(t_0)+E_F(u)^2\mid u\in\F_q,\ \Tr_{q/2}(uR_{F,t_0}(u))=n+1\,\}.
\end{align*}

\end{enumerate}
\end{theorem}

\begin{remark}
    Theorem \ref{easycase} can be viewed as the special case of Theorem \ref{exmax} where $q_1=q^{1/4}$. In fact, when $W_F^\ast\subset \F_{q^{1/4}}$, one can show that $t_0\in F(\F_q)$  using Lemma \ref{kerW_F}.  
\end{remark}

\begin{corollary}\label{exmaxc}
    Assume that $F$ satisfies conditions \textup{(i)--(iv)} and that $W_F^\ast\subset \F_{\sqrt{q}}$. 
    Let $t_0\in \F_q$ be an element satisfying $t_0^{\sqrt{q}}+t_0=1$. 
    Then the curve $D_{F,t_0}$ is $\F_q$-maximal. 
\end{corollary}
\begin{proof}
The assertion follows from Theorem \ref{exmax}(iii) applied to $q_1=\sqrt{q}$. 
\end{proof}

To prove Theorem \ref{exmax}, we need the following auxiliary lemma. 
\begin{lemma}\label{Trt_0}
Let $r=2^s$ be a power of $2$. 
Let $t_0\in\F$ be an element satisfying $t_0^{r}+t_0=1$. The following hold: 
\begin{enumerate}
\item We have $t_0\in\F_{r^2}$ and $t_0^{r+1}\in \F_{r}$. 
\item We have $\Tr_{r/2}(t_0^{r+1})=1$. 
\item Let $\Tr_{r^2/2}\colon W_2(\F_{r^2})\to W_2(\F_2)$ be the trace map. Then 
\[\Tr_{r^2/2}(t_0,0)=s(1,0)+(0,1). \]
\end{enumerate}
\end{lemma}
\proof{(i) 
The first assertion follows from 
$t_0^{r^2}=(t_0+1)^{r}=t_0$. 
The second one follows since 
$(t_0^{r+1})^{r-1}=t_0^{r^2-1}=1$. 

(ii) Suppose to the contrary that $\Tr_{r/2}(t_0^{r+1})=0$. Then there exists $x\in \F_{r}$ such that  $t_0^{r+1}=x^2+x$. Since $t_0^{r}=t_0+1$, we have 
\[
t_0^2+t_0=t_0(t_0+1)=t_0^{r+1}=x^2+x,
\]
which implies $t_0\in\F_{r}$, a contradiction. 

(iii) By Lemma \ref{ij}, we have 
 \[\Tr_{r^2/2}(t_0,0)=
 \sum_{i=0}^{2s-1} (t_0^{2^i},0)=
 \left(\Tr_{r^2/2}(t_0),\, \sum_{0\leq i<j\leq2s-1}t_0^{2^i}t_0^{2^j}\right). 
 \]
  Note that the coefficient of $x^{2s-1}$ in 
  $\Nr_{r^2/2}(x+t_0):=\prod_{i=0}^{2s-1}(x+t_0^{2^i})$ equals the first component above, and that of $x^{2s-2}$ equals the second one. Since 
  $t_0^r+t_0=1$, we compute 
  \begin{align*}
  \Nr_{r^2/2}(x+t_0)&=\Nr_{r/2}\Nr_{r^2/r}(x+t_0)\\
  &=\Nr_{r/2}(x^2+x+t_0^{r+1})=\prod_{i=0}^{s-1}\left(x^2+x+t_0^{2^i(r+1)}\right). 
  \end{align*}
  Hence  
  \[\Tr_{r^2/2}(t_0,0)=\left(s, \frac{s(s-1)}{2}+\Tr_{r/2}(t_0^{r+1})\right). \]
  The claim then follows from $s(1,0)=(s, \frac{s(s-1)}{2})$ and~(ii).  
\qed}

\begin{proof}[(Proof of Theorem \ref{exmax})]
(i) The assertion follows from Lemma \ref{Trt_0}(i).

(ii) We  prove $S_F=\{t_0+F(u)\mid u\in\F_q\}$. The assertion for $T_{R,\max} \cup T_{R,\min}$ then follows from Theorem \ref{mainthm}(ii). By the assumption $W_F^\ast\subset \F_{q_1}$, for every $u\in W_F^\ast$, we have 
\begin{align*}
Q_q(u)=\xi_2\bigl(\Tr_{q_1/2}(2n(u,0))\bigr)=\psi_2\bigl(\Tr_{q_1/2}(nu^2)\bigr)=\psi_{q_1}(nu). 
\end{align*}
Since $q=q_1^{2n}$ and $\Tr_{q_1^2/q_1}(t_0)=t_0^{q_1}+t_0=1$, 
we have 
\[
\Tr_{q/q_1}(t_0)=\Tr_{q_1^2/q_1} \circ \Tr_{q/q_1^2}(t_0)=
n\Tr_{q_1^2/q_1}(t_0)=n.
\]
Hence, we obtain  
\[\psi_q(t_0u)=\psi_{q_1}(\Tr_{q/q_1}(t_0)u)=\psi_{q_1}(nu).\]
 Thus we have $t_0\in S_F$, which proves the assertion by Lemma \ref{kerW_F}. 

(iii) 
Let $u\in \F_q$. Then we have $t_0+F(u)\in S_{F,-}$ if and only if 
\[Q_q(t_0+F(u))=-(\sqrt{-1})^{[\F_q:\F_2]/2}. \]
Set $q_1 = 2^s$. Using the relation $ns = [\F_q:\F_2]/2$ together with Lemma \ref{Trt_0}(iii), we compute
\begin{align*}
Q_q(t_0) &= \xi_2 \circ \Tr_{q_1^2/2} \circ \Tr_{q/q_1^2}(t_0,0) \\
         &= \xi_2\bigl(n \Tr_{q_1^2/2}(t_0,0)\bigr) \\
         &= \xi_2\bigl(ns(1,0) + n(0,1)\bigr) \\
         &= (-1)^n (\sqrt{-1})^{[\F_q:\F_2]/2}.
\end{align*}
Thus we obtain 
\begin{align*}
Q_q(t_0+F(u))&=Q_q(t_0)Q_q\bigl(F(u)\bigr)\psi_q\bigl(t_0F(u)\bigr)\\
&=(-1)^n(\sqrt{-1})^{[\F_q:\F_2]/2}\psi_q\bigl(uR_{F,0}(u)\bigr)\psi_q\bigl(F^\ast(t_0)^2u^2\bigr)&\\&=(-1)^n(\sqrt{-1})^{[\F_q:\F_2]/2}\psi_q(uR_{F,t_0}(u)). 
\end{align*}
Here, in the second equality, we use Lemmas \ref{b4} and \ref{F0}.  This proves the assertion for $S_{F,-}$. That for $T_{R,\max}$ follows from this by applying Theorem \ref{mainthm}(ii). 
\end{proof}

\subsection{Examples of Theorem~\ref{exmax}}\label{EMT}

We give some examples to which Theorem~\ref{exmax} applies.
Fix a  polynomial 
\[f(x)=\sum_{i=0}^eb_ix^{i}\in\F_p[x]\] 
satisfying the following properties: 
\begin{enumerate}
\item $b_0b_e\neq0$.
\item $f(1)=0$. 
\item All the roots of $f(x)$ are simple.  
\end{enumerate}

\begin{example}
Let $c_0=1,c_1,\dots,c_{e-1}\in\F_p\setminus\{0\}$ be distinct elements. Then $f(x):=\prod_{i=0}^{e-1}(x+c_i)$ satisfies the above properties. 
\end{example}

Set  $g(x):=x^ef(x)f(x^{-1})$. 
\begin{lemma}
\label{coeff}
The polynomial $g(x)$ can be written as 
\[g(x)=a_ex^{2e}+a_{e-1}x^{2e-1}+\cdots+a_1x^{e+1}+a_1x^{e-1}+\cdots+a_{e-1}x+a_e
\]
with $a_i\in\F_p$. 
\end{lemma}
\proof{The polynomial $g(x)$ satisfies $g(x)=x^{2e}g(x^{-1})$. Thus one can write 
\[g(x)=a_ex^{2e}+a_{e-1}x^{2e-1}+\cdots+a_1x^{e+1}+a_0x^e+a_1x^{e-1}+\cdots+a_{e-1}x+a_e.\]
We have $a_0=g(1)=f(1)^2=0$. 
\qed}

\begin{lemma}\label{n/2odd}
Let $n$ be the smallest integer $\geq1$ such that the polynomial $g(x)$ divides $x^n+1$. 
 Then  $n$ is even, and $n/2$ is odd. 
  \end{lemma}
\proof{Let $\alpha_1,\dots,\alpha_e$ be the roots of $f(x)$ in $\F$, and let $n_0$ be the least common multiple of the orders of $\alpha_i$ in the multiplicative group $\F^\times$. Note that $n_0$ is odd, since $\F^\times$ does not have elements of even order. 

We claim that $n=2n_0$. Indeed, since $\alpha_i^n+1=0$ for each $i$, the integer $n$ must be a multiple of $n_0$. Moreover, 
since $g(1)=g'(1)=0$, the polynomial $g(x)$ is not separable. Hence $x^n+1$ is not separable, which implies  that $n$ is even. 

Therefore, $n$ must be a multiple of $2n_0$. 
On the other hand, by assumption, both $f(x)$ and $x^ef(x^{-1})$ have only simple roots. By the choice of $n_0$, the polynomial $x^{n_0}+1$ is divisible by both $f(x)$ and $x^ef(x^{-1})$. Hence, $g(x)$ divides $x^{2n_0}+1$. 

This shows that $n\leq 2n_0$. Since $n$ is a multiple of $2n_0$, the claim follows. 
 \qed}

\medskip

Consider the $\F_p$-linearized polynomial
\[
R(x):=a_ex^{p^e}+a_{e-1}x^{p^{e-1}}+\cdots+a_1x^p.
\]
Here, $a_1,\ldots,a_e \in \F_p$ are the coefficients of
$g(x):=x^ef(x)f(x^{-1})$ 
as in Lemma \ref{coeff}. 
We shall use the same notation as in
Section \ref{max2}.
We also write $R$ for $\sum_{i=1}^e a_i \tau^i \in \F_p\{\tau^{\pm1}\}$. 
Set $E:=R+R^\ast$. 
Note that $E=\tau^{-e}g(\tau)$. 

Let $q$ be a power of $p$. 
Take an element $a\in \F_q$ and consider the $\F_q$-curve
associated with $R_a(x) := R(x) + ax$
\begin{equation}\label{exf}
C_a\colon y^p + y = x R_a(x). 
\end{equation} 

In the following theorem, we give a necessary and sufficient condition for the curve $C_a$ to be $\F_q$-extremal or $\F_q$-maximal.
We note that Theorem \ref{exRx} may be regarded as a characteristic $2$ analogue of Tatematsu's results in odd characteristic; see \cite[Theorem 1.1]{ITT}. 

\begin{theorem}\label{exRx}
Let $n$ be the integer from Lemma \ref{n/2odd}. Define  $\widetilde{\alpha}(u)\in\F_p[u^{p^{-\infty}}]$ by 
\[
\widetilde{\alpha}(u):=\sqrt{R(1)}+f'(1)+E(u). 
\]
Then the following hold: 
\begin{enumerate}
\item If there exists $a\in \F_q$ such that $C_a$  is $\F_q$-extremal (in the sense of Definition \ref{mmx}), then  $\F_{p^n}\subset \F_q$. 
\item Assume that $\F_{p^n}\subset \F_q$, and set $q=p^{ns}$. Then the curve $C_a$ is $\F_q$-extremal if and only if $a=\widetilde{\alpha}(u)^2$ for some $u\in\F_q$. 
Moreover, $C_a$ is $\F_q$-maximal if and only if 
$a=\widetilde{\alpha}(u)^2$ for some $u\in \F_q$ satisfying 
\[
\Tr_{q/2}(uR(u)+\widetilde{\alpha}(0)u)=s+1.
\]
\end{enumerate}
\end{theorem}
\proof{(i) Let $V$ be the kernel of $E\colon \F\to \F$. 
By Lemma \ref{basic}(iii), $n$ is the smallest integer such that $V\subset \F_{p^n}$. Thus the condition $\F_{p^n}\subset \F_q$ is necessary for $C_a$ to be $\F_q$-extremal by Corollary \ref{exN}. 

\medskip

\noindent
(ii) From now on, we assume $\F_{p^n}\subset \F_q$, and 
write $q=p^{ns}$. Let $n_0=n/2$. 
Set $F:=f(\tau)=\sum_{i=0}^eb_i\tau^i$. Then $F^\ast=f(\tau^{-1})$, since each $b_i$ belongs to $\F_p$. Hence, $F^\ast F=E$. We use the notation from Subsection \ref{set}. 

We first show that $W_F^\ast\subset \F_{p^{n_0}}$. Indeed, since $g(x)$ divides $x^n+1$, the factor $x^ef(x^{-1})$ divides $x^n+1=(x^{n_0}+1)^2$. However, the polynomial $x^ef(x^{-1})$ has distinct roots, and hence divides $x^{n_0}+1$. This implies that one can write 
\[\tau^{n_0}+1=hF^\ast\]
for some $h\in\F_p\{\tau^{\pm1}\}$, which is equivalent to $W_F^\ast\subset \F_{p^{n_0}}$ by Lemma \ref{basic}(iii). 

Fix an element $t_0\in \F_{p^2}\subset \F_{p^n}$ satisfying $t_0^{p}+t_0=1$. Since $n_0$ is odd by Lemma \ref{n/2odd}, 
\[t_0^{p^{n_0}}+t_0=\sum_{i=0}^{n_0-1}(t_0^p+t_0)^{p^i}=n_0=1. \]
Therefore, we can apply Theorem \ref{exmax} to 
$(F,t_0,q_1)=(F,t_0,p^{n_0})$ to deduce that 
\begin{align*}
&T_{R,\max} \cup T_{R,\min}=\alpha_F(t_0)+E(\F_q)^2, \\
&T_{R,\max}=\{\,\alpha_F(t_0)+E(u)^2\mid u\in\F_q,\, \Tr_{q/2}(uR_{F,t_0}(u))=s+1\,\}. 
\end{align*}
It remains to prove that $\alpha_F(t_0)=\widetilde{\alpha}(0)^2=R(1)+f'(1)^2$. 
Since $b_i\in\F_p$, the term $\gamma_F$ is equal to 
 \[\sum_{0\leq i<j\leq e}b_ib_j=R(1). 
 \]
 Since $t_0^{p^{-1}}=t_0+1$, we have $t_0^{p^{-i}}=t_0+i$ for every integer $i \ge 0$. Thus,  
 \[
 F^\ast(t_0)=\sum_{i=0}^e b_i t_0^{p^{-i}}=t_0\sum_{i=0}^eb_i+\sum_{i=0}^eib_i. 
 \]
The term $\sum_{i=0}^eb_i=f(1)$ is  zero by assumption, and the term $\sum_{i=0}^eib_i$ is equal to $f'(1)$. Hence  
\begin{equation}\label{Ft0}
F^\ast(t_0)=f'(1). 
\end{equation}
This proves 
\[\alpha_F(t_0)=R(1)+f'(1)^2,  \]
and the claim follows. 
\qed}

\medskip

In the following proposition,
we consider an affine curve of the form 
\begin{equation}\label{xpax}
X_a\colon y^p+y=x^{p+1}+ax^2
\end{equation} 
for general $a \in \F_q$. When $a=0$, the smooth compactification of $X_0$ is $\F_{p^2}$-isomorphic to the Hermitian curve defined by the homogenous equation 
\[
X^{p+1}+Y^{p+1}+Z^{p+1}=0
\]
in $\mathbb{P}^2_{\F_{p^2}}$. 
This curve is known to be $\F_{p^2}$-maximal \cite[Example 6.3.6]{St}. Note that $X_a$ is 
a particular case of \eqref{exf}
by taking $f(x) = x+1$ and $R_a(x) = x^p + a x$.

\begin{proposition}\label{Hermitian curve}
Let $p$ be a power of $2$, and let $q=p^m$ be a power of $p$.   We assume that $m$ is even, and write $m=2m_0$. 

Let $a\in \F_q$, and define $\alpha:=\Tr_{q/p^2}(a)$.  Consider the curve $X_a$ defined in \eqref{xpax}. Then the following hold: 
\begin{enumerate}
\item If $\alpha\neq0$, then $X_a$ is not $\F_q$-extremal (in the sense of Definition \ref{mmx}). The ${\rm Fr}_q$-eigenvalues are 
\[\pm(\sqrt{-1})^{t(\alpha^{p+1})}2^{[\F_q:\F_2]/2}, \]
each with the same multiplicity. Here, $t(\alpha^{p+1})$ denotes the integer in $\{0,1\}$ congruent to $\Tr_{p/2}(\alpha^{p+1})$ modulo $2$. 
\item If $\alpha=0$, then the curve $X_a$ is $\F_q$-extremal. Furthermore, $X_a$ is $\F_q$-maximal if and only if 
\[m_0+\Tr_{p/2}(\beta)=1,\]
 where $\beta\in\F_p$ is given by 
\[
\beta:=\sum_{\substack{0\leq i<j\leq m-1 \\ 2\mid j,\ 2\nmid i}}a^{p^i}a^{p^j}. 
\]
\end{enumerate}
\end{proposition}
\proof{Set $R(x)=x^p$. The kernel of $E=\tau+\tau^{-1}$ is  $\F_{p^2}$.  Note that for $x\in \F_{p^2}$, the element $x^{p+1}$ lies in $\F_p$. Therefore, for every $x\in\F_{p^2}$, we have 
\begin{equation}\label{trxpax}
\Tr_{q/p}(xR(x)+ax^2)=2m_0x^{p+1}+\Tr_{p^2/p}(\Tr_{q/p^2}(a)x^2)
=\Tr_{p^2/p}(\alpha x^2). 
\end{equation}
By \eqref{N0pm}, the curve $X_a$ is $\F_q$-extremal if and only if the form above vanishes identically on $\F_{p^2}$, which is equivalent to $\alpha=0$. This proves the first parts of (i) and (ii). 

\medskip 

Assume that $\alpha\neq0$. In this case, set 
\[
F:=\alpha^{(p-1)/4}\tau+\alpha^{(1-p)/4}=\alpha^{-(p+1)/4}(\tau+1)\alpha^{1/2}. \]
Since $\alpha^{-(p+1)/4}\in\F_p$, we have $F^\ast=\alpha^{(1-p)/4}(\tau^{-1}+1)$. Hence, 
\[
 F^\ast(1)=0,\qquad W_F^\ast=\F_p,\qquad F^\ast F=E. 
\]

Furthermore, we have 
\[
\alpha_F(t)=\alpha^{(1-p)/2}+F^\ast(t)^2=\alpha^{(1-p)/2}(t^{p^{-1}}+t+1)^2. \]
We have $W_F=\alpha^{-1/2}\F_p$. Thus, the form \eqref{trxpax} vanishes on $W_{F,q}=W_F$. 
 Therefore, by Proposition \ref{Ima}, we can write 
\begin{equation}\label{a=alpha t}
    a=\alpha_F(t)=\alpha^{(1-p)/2}(t^{p^{-1}}+t+1)^2
\end{equation}
for some $t\in \F_q$. 

We apply Theorem \ref{FEV1} to compute the ${\rm Fr}_q$-eigenvalues. For $v\in W_F^\ast=\F_p$, we have 
\[Q_q(t+v)Q_q(t)^{-1}=\psi_q(tv)Q_q(v). 
\]
Since $X_a$ is not $\F_q$-extremal, the ratio above is not identically $1$. Consequently, by Lemma \ref{ghom}, half of the eigenvalues are given by \[Q_q(t)^{-1}(-1-\sqrt{-1})^{[\F_q:\F_2]},\] 
and the other half  by its negative. Note that 
\[Q_q(t)=\xi_2\bigl(\Tr_{q/2}(t),\ast\bigr)=\pm\xi_2\bigl(\Tr_{q/2}(t),0\bigr). \] 
Therefore, to complete the proof of~(i), it suffices to compute $\Tr_{q/2}(t)$. Using \eqref{a=alpha t}, 
we compute   
\begin{align*}
\Tr_{q/p}(t)&=\sum_{i=0}^{m_0-1}(t^{p^{-1}}+t)^{p^{-2i}}\\
&=\sum_{i=0}^{m_0-1}(a\alpha^{(p-1)/2})^{p^{-2i}/2}+m_0=\alpha^{(p+1)/4}+m_0. 
\end{align*}
Thus, 
\[\Tr_{q/2}(t)=\Tr_{p/2}(\alpha^{p+1})+m_0[\F_p:\F_2],\]
and the remaining part of~(i) follows from 
$(-1-\sqrt{-1})^{[\F_q:\F_2]}=(\sqrt{-1})^{m_0[\F_p:\F_2]}\cdot 2^{[\F_q:\F_2]/2}$. 

\medskip

Now assume that $\alpha=0$. 
Set $f(x):=x+1$. Then 
\[g(x)=xf(x) f(x^{-1})=x^2+1,\qquad E=\tau^{-1} g(\tau)=\tau+\tau^{-1}. \]
Thus, the curve $X_a$ is the special case of \eqref{exf} where $f(x)=x+1$. Since 
$\widetilde{\alpha}(u)=E(u)$ and 
$\widetilde{\alpha}(0)=0$,  
Theorem \ref{exRx}(ii) implies that 
the curve $X_a$ is $\F_q$-maximal if and only if there exists $u\in \F_q$ satisfying $a=u^{p^2}+u$ and 
\[\Tr_{q/2}(u^{p+1})=m_0+1.
\]
Therefore, it suffices to prove $\Tr_{q/p}(u^{p+1})=\beta$ for any $u$ satisfying $a=u^{p^2}+u$. We have 
\[u^{p+1}+(u^{p+1})^p=u^{p+1}+(a+u)u^p=au^p. \]
Hence, since $\alpha=\Tr_{q/p^2}(a)=0$, 
\begin{align*}
\Tr_{q/p}(u^{p+1})=\sum_{i=0}^{m_0-1}(au^p)^{p^{2i}}
&=\sum_{i=0}^{m_0-1}a^{p^{2i}}(a+a^{p^2}+\cdots+a^{p^{2(i-1)}}+u)^p \\
&=\beta+\Tr_{q/p^2}(a) u^p=\beta. 
\end{align*}
Thus the assertion follows. 
\qed}

\subsection{Periods and Parities}\label{periods}
Let $p$ be a power of $2$. 
Let $\overline{C}$ be a geometrically connected smooth projective curve over $\mathbb{F}_p$. Assume that it is supersingular; equivalently, assume that $\overline{C}$ is $\mathbb{F}_{p^N}$-minimal for some integer $N \geq 1$. 

In this subsection, we consider the following notion. 
\begin{definition}
    \begin{enumerate}
        \item The \emph{$\F_p$-period} $\mu(\overline{C})$ of $\overline{C}$ is the smallest integer $n \geq 1$ such that $\overline{C}$ is $\mathbb{F}_{p^n}$-extremal.
        \item The \emph{$\F_p$-parity} $\delta(\overline{C})$ of $\overline{C}$ is defined to be $-1$ if $\overline{C}$ is $\F_{p^{\mu(\overline{C})}}$-maximal,  and $1$ if $\overline{C}$ is $\F_{p^{\mu(\overline{C})}}$-minimal\footnote{See the footnote in \S\ref{Periods and Parities}}. 
    \end{enumerate}
\end{definition}
 A similar notion was introduced for abelian varieties in \cite[Definition 4.1]{KP} and \cite[Section 2]{SX}. 

 Let 
 \[(\mu,\delta)\in \Z_{\geq1}\times\{\pm1\}. \]
 We consider the following problem: for which pairs $(\mu,\delta)$ does there exist an $\F_p$-linearized polynomial 
 \[
 R(x)=\sum_{i=0}^ea_ix^{p^i}\in \F_p[x], 
 \]
 where $e\geq1$ and $a_e\neq0$, such that 
 \[(\mu,\delta)=(\mu(\overline{C}_R),\delta(\overline{C}_R))?\] 

We begin with some well-known cases. 
\begin{proposition}\label{well known case for period}
Let $\mu$ be an integer $\geq1$. Then the following hold: 
\begin{enumerate}
\item If $\mu$ is odd, then there exists no $\F_p$-linearized polynomial $R(x)$ as above such that the $\F_p$-curve $\overline{C}_R$ has $\F_p$-period $\mu$. 
\item Assume that $\mu$ is even, and write $\mu=2m$. Then the $\F_p$-curve 
\[
C_R\colon y^p+y=x^{p^m+1}
\]
satisfies 
\[
(\mu,-1)=(\mu(\overline{C}_R),\delta(\overline{C}_R)). 
\]
\end{enumerate}
\end{proposition}
\begin{proof}
    (i) The assertion follows from Corollary \ref{exN} and Lemma \ref{even}. 

    (ii) Observe that the curve $C$ is a quotient of the curve 
    \[
    D\colon y^{p^m}+y=x^{p^m+1}  
    \]
    via the morphism 
    \[
    D\to C_R,\quad (x,y)\mapsto \left(x,\sum_{i=0}^{m-1}y^{p^i}\right). 
    \]
    By Proposition~\ref{Hermitian curve}(ii) applied to $p=p^m$ and $a=0$, we see that $D$ is $\F_{p^\mu}$-maximal. Hence, its quotient $C_R$ is also $\F_{p^\mu}$-maximal. 

    It remains to prove $\mu=\mu(\overline{C}_R)$. Let $N$ be a positive integer such that $C_R$ is $\F_{p^N}$-extremal. Then, by Corollary \ref{exN}, we must have $V=\F_{p^{\mu}}\subset \F_{p^N}$. This proves $\mu\leq N$, and hence $\mu=\mu(\overline{C}_R)$. 
\end{proof}

It remains to treat the case where $\mu$ is even and $\delta=1$. In this case, it is not true in general that there exists some $R$ for which 
\begin{equation}\label{mu-1=}
    (\mu,1)=(\mu(\overline{C}_R),\delta(\overline{C}_R)). 
\end{equation}
\begin{lemma}
    The case $\mu=2$ is impossible. Furthermore, when $p=2$, the case $\mu=4$ is also impossible. 
\end{lemma}
\begin{proof}
    Let $\mu=2$, and suppose that there exists $R$ for which \eqref{mu-1=} holds. 
    Since $\overline{C}_R$ is $\F_{p^2}$-extremal, we have $V:=\ker(R+R^\ast)\subset \F_{p^2}$. Since $\dim_{\F_p}V\geq2$, we must have $\dim_{\F_p}V=2$. Hence, $R(x)=ax^p+bx$ for some $a\in \F_p^\times$ and $b\in\F_p$. 

    In the field $\F_{p^2}$, one can write $a=c^{p+1}$. Then, under the change of coordinates $(x,y)\mapsto (cx,y)$, the curve $C_R$ is $\F_{p^2}$-isomorphic to $C_{R_1}$ where $R_1(x)=x^p+b_1x$. However, by Proposition~\ref{Hermitian curve}(i), the curve $C_{R_1}$ cannot be $\F_{p^2}$-minimal. 

    \medskip 
    
    \noindent
 Assume now that $\mu=4$ and $p=2$, and suppose that there exists $R$ for which \eqref{mu-1=} holds. Since $V\subset \F_{p^4}$ and $\F_p=\{0,1\}$, there are three possibilities for $R$: 
\begin{align*}
x^p, \quad x^p+x,\quad x^{p^2}+ax\qquad(a=0,1). 
\end{align*}
By Proposition~\ref{Hermitian curve}, we have $\delta(\overline{C}_R)=-1$ in the first two cases. 

It remains to consider the third case. 
Set $F:=\tau^2+1$. Then $R=R_F$. Let $q=p^4$. Using the fact that $W_F^\ast=\F_{p^2}$, one can show that  
\[
S_F=\{t\in\F_{p^4}\mid t^{p^2}+t=1\}. 
\]
By Lemma~\ref{Trt_0}(iii), it follows that $S_F=S_{F,-}$, and hence $S_{F,+}=\emptyset$. Therefore, $C_R$ is not $\F_{p^4}$-minimal. 
\end{proof}

The following proposition, together with the lemma above, gives a complete answer in the case where $\mu$ is a multiple of $4$. 
\begin{proposition} \label{not well known case for period}
Assume that $\mu=2m$ where $m$ is even. Then, in each of the following cases, the given $R(x)$ satisfies 
\[
(\mu,1)=(\mu(\overline{C}_R),\delta(\overline{C}_R)). 
\]
\begin{enumerate}
    \item If $m=2$ and $\F_p\neq\F_2$, then 
    \[
    R(x):=x^p+ax, 
    \]
    where $a\in \F_p^\times$ satisfies $\Tr_{p/2}(a)=0$. 
\item If $m\geq4$, then 
    \[
R(x):=x^p+x^{p^3}+\cdots+x^{p^{m-1}}=\sum_{i=1}^{m/2}x^{p^{2i-1}}. 
\]
\end{enumerate}
\end{proposition}
\begin{proof}
(i) In this case, by Proposition~\ref{Hermitian curve}(i) applied to $q=p^2$, the curve $C_R=X_a$ is not $\F_{p^2}$-extremal, and the ${\rm Fr}_{p^2}$-eigenvalues are given by 
\[
\pm 2^{[\F_{p}:\F_2]}. 
\]
Therefore, $C_R$ is $\F_{p^4}$-minimal. This proves 
\[
(\mu(\overline{C}_R),\delta(\overline{C}_R))=(4,1). 
\]

\medskip

\noindent
(ii) In this case, define 
\[
E:=R+R^\ast,\qquad V:=\ker E. 
\]
Let $N$ be the smallest integer $\geq1$ such that 
\[
V\subset \F_{p^N}. 
\]
We first claim that $\mu\leq N$. Indeed, by comparing dimensions, we have 
\[
2(m-1)=\dim_{\F_p}
V\leq N.\]
On the other hand, by Lemma~\ref{even}, $N$ must be even. Hence, it suffices to show that $2(m-1)<N$. 

Suppose, to the contrary, that $2(m-1)=N$. Then, we must have $V=\F_{p^N}$. By Lemma~\ref{basic}(iii), it follows that 
\[
E=a\tau^M(\tau^N+1)
\]
for some $a\in \F^\times$ and $M\in\Z$. However, this is impossible, since $E$ has more than two nonzero coefficients. 
Thus, we conclude that 
\[\mu=2m\leq N. \]

By Corollary~\ref{exN}, for every integer $n\geq1$, extremality of $C_R$ over $\F_{p^n}$ implies that $V\subset \F_{p^n}$. Hence, to prove the assertion, it suffices to show that $C_R$ is $\F_{p^\mu}$-minimal. 

\medskip 

We show that $C_R$ is $\F_{p^\mu}$-minimal. Set $e:=m-1$ and $q:=p^\mu$. Define 
    \[
    F:=\sum_{i=0}^{m-1}\tau^i\in\F_p\{\tau^{\pm1}\}. 
    \]
    Since $m$ is even, we have $F^\ast(1)=0$. We claim that 
    \begin{equation}\label{WF=kerTr}
        W_F^\ast:=\ker F^\ast=\ker(\Tr_{p^m/p}\colon \F_{p^m}\to \F_p).
    \end{equation}
Indeed, the inclusion $\supset$ is clear from a direct computation, and then the equality follows by comparing cardinalities. 

Consequently, the element $F$ satisfies conditions~\textup{(i)--(iii)} stated after \eqref{Fdef}. Observe that 
    \[
    R+R^\ast=F^\ast F. 
    \]
Moreover, 
\[
\alpha_F(t)=\frac{m(m-1)}{2}+F^\ast(t)^2. 
\]
    Let $t_0\in\F_{p^2}$ be such that $t_0^{p}+t_0=1$. Then 
    \[
    F^\ast(t_0)=\sum_{i=0}^{m-1}t_0^{p^{-i}}=\sum_{i=0}^{m-1}(t_0-i)=\frac{m(m-1)}{2}, 
    \]
    and hence $\alpha_F(t_0)=0$. Thus, we obtain $D_{F,t_0}=C_R$. 

Therefore, by Theorem~\ref{FEV1}(iv), the ${\rm Fr}_q$-eigenvalues of $\overline{C}_R$ are given by 
\[
Q_q(t_0+v)^{-1}(\sqrt{-1})^{[\F_q:\F_2]/2}\cdot 2^{[\F_q:\F_2]/2}\qquad (v\in W_F^\ast). 
\]
It remains to show that $Q_q(t_0+v)=(\sqrt{-1})^{[\F_q:\F_2]/2}$ for all $v\in W_F^\ast$. 
We compute 
\[
Q_q(t_0+v)=\psi_q(t_0v)Q_q(t_0)Q_q(v). 
\]
By \eqref{WF=kerTr} and $t_0\in\F_{p^2}$, it follows that  
\begin{align*}
\psi_q(t_0v)&=1,\\
    Q_q(t_0)&=Q_{p^2}(t_0)^m=\bigl(-(\sqrt{-1})^{[\F_p:\F_2]}\bigr)^m,\qquad(\text{Lemma~\ref{Trt_0}(iii)})\\
    Q_q(v)&=\xi_{p^m}(\Tr_{p^{2m}/p^m}(v,0))=\xi_{p^m}(0,v^2)=\psi_p\bigl(\Tr_{p^m/p}(v)^2\bigr)=1. 
\end{align*}
Since $m$ is even, we obtain $Q_q(t_0+v)=(\sqrt{-1})^{[\F_q:\F_2]/2}$, and the curve $C_R$ is $\F_{p^\mu}$-minimal. This completes the proof. 
\end{proof}

When $\mu=2m$ with $m\geq1$ odd, we have the following result. 
\begin{proposition}\label{2odd impossible}
Let $m\geq1$ be an odd integer, and let 
\[
R(x)=\sum_{i=0}^ea_ix^{p^i}\in \F_p[x], 
\]
where $e\geq1$ and $a_e\neq0$. Assume that there exists an element 
\[F\in \F_p\{\tau^{\pm1}\} 
\]
as in \eqref{Fdef} such that 
\begin{itemize}
    \item $F$ satisfies conditions \textup{(i)--(ii)} stated thereafter, 
    \item $W_F=\ker F\subset \F_{p^m}$, and 
    \item $R+R^\ast=F^\ast F$. 
\end{itemize}
Then $C_R$ is not $\F_{p^{2m}}$-minimal. In particular, 
\[
(\mu(\overline{C}_R),\delta(\overline{C}_R))\neq (2m,1)
\]
for such $R$. 
\end{proposition}
\begin{proof}
Set $q:=p^{2m}$. Suppose, for contradiction, that $C_R$ is $\F_{p^{2m}}$-minimal. Then, by Corollary~\ref{exN}, we have 
\[
\ker (R+R^\ast)\subset \F_{p^{2m}}. 
\]
Hence, $F$ satisfies conditions \textup{(i)--(iv)} stated after \eqref{Fdef}, and thus Theorem~\ref{mainthm} applies. Consequently, there exists $t\in S_{F,+}$ such that 
\[
a_0=\alpha_F(t). 
\]

Note that, by Lemma~\ref{kfin}(i), the second assumption on $F$ is equivalent to 
\begin{equation}\label{WF in pm}
    W_F^\ast:=\ker F^\ast \subset \F_{p^m}. 
\end{equation}
Hence, for every $v\in W_F^\ast$, 
\[
Q_q(v)=\psi_{p^m}(v). 
\]
On the other hand, since $t\in S_{F}$, this is equal to 
\[
\psi_q(tv)=\psi_{p^m}((t^{p^m}+t)v). 
\]
Thus, by Lemma~\ref{kerW_F}, there exists $s\in \F_{p^m}$ such that 
\[
t^{p^m}+t+1=F(s). 
\]
By \eqref{WF in pm}, there exists $g\in \F_p\{\tau^{\pm1}\}$ such that 
\[
x^{p^m}+x=g(F^\ast(x)). 
\]
Since $\alpha_F(t)=\gamma_F+F^\ast(t)^2=a_0\in\F_p$ and $\gamma_F\in \F_p$, we have 
\[
F(s)=g(F^\ast(t))+1\in\F_p, 
\]
and hence 
\[
s\in W':=\ker\bigl((\tau+1)F\bigr)\cap \F_{p^m}. 
\]

Assume first that $F(s)=0$. In this case, $t^{p^m}+t=1$, and Lemma~\ref{Trt_0}(iii) implies that $Q_q(t)=-(\sqrt{-1})^{[\F_q:\F_2]/2}$. This contradicts the assumption that $t\in S_{F,+}$. 

Assume now that $F(s)\in\F_p^\times$. In this case, $W'$ is strictly larger than $W_F$. On the other hand, $W_F$ has codimension $1$ in $\ker\bigl((\tau+1)F\bigr)$. Thus, 
\[
W'=\ker\bigl((\tau+1)F\bigr), 
\]
and hence 
\[
\ker\bigl((\tau+1)F\bigr)\subset \F_{p^m}. 
\]
By $F(x) \in \F_p[x]$, we have 
$F(1)=F^\ast(1)=0$. 
Since $F(x)$ is $\F_p$-linearized, we have 
$\tau+1\mid F$. Hence,  
\[
\F_{p^2}\subset \ker\bigl((\tau+1)F\bigr)\subset \F_{p^m}. 
\]
This is impossible because $m$ is odd. 
\end{proof}

When $\F_p$ contains a primitive $m$-th root of unity, the polynomial $R$ always admits such a factorization. 
\begin{proposition}\label{2odd impossible 2}
Let $m\geq1$ be an integer such that the polynomial $x^m+1$ splits completely over $\F_p$. 
Let 
\[
R(x)=\sum_{i=0}^ea_ix^{p^i}\in\F_p[x], 
\]
  where $e\geq1$ and $a_e\neq0$. Assume that 
  \[
  V:=\ker(R+R^\ast)\subset \F_{p^{2m}}. 
  \]
Then $R$ admits an element $F$ satisfying the conditions in Proposition~\ref{2odd impossible}. 
\end{proposition}
\begin{proof}
Since $R$ has coefficients in $\F_p$, the Frobenius ${\rm Fr}_p$ on $\F$ restricts to an $\F_p$-linear automorphism $\varphi$ of $V$. 
By Corollary~\ref{dec2}, it suffices to find a maximal totally isotropic subspace $W\subset V$ such that $\varphi(W)=W$ and $W\subset \F_{p^m}$. 

Define $N:=\varphi^m+{\rm id}_V$. We claim that 
\[
(\ker N)^\perp\subset\ker N. 
\]
Indeed, for every $v\in (\ker N)^\perp$ and every $u\in V$, 
\[
\omega(Nv,u)=\omega(v,Nu). 
\]
Since $V\subset\F_{p^{2m}}$, we have $N^2=0$, 
and hence 
\[
\Ima N\subset\ker N. 
\]
Since $v \in (\ker N)^{\perp}$ and 
$Nu \in \ker N$, we obtain $\omega(Nv,u)=\omega(v,Nu)=0$. Since $u$ is arbitrary and $\omega$ is nondegenerate, it follows that $Nv=0$, proving the claim. 

Set $X:=\ker N/(\ker N)^\perp$. Then the form $\omega$ naturally induces a nondegenerate symplectic form $\omega_X$ on $X$. Moreover, $\varphi$ naturally induces an $\F_p$-linear automorphism $\varphi_X$ of $X$ by Lemma \ref{b5}(iv). 
By Lemma~\ref{mti stable} below, $X$ admits a maximal totally isotropic subspace $W'$ such that $\varphi_X(W')=W'$. Since $\ker N=V\cap\F_{p^m}$, the inverse image $W$ of $W'$ in $\ker N$ satisfies 
the required properties. 

This completes the proof. 
\end{proof}
\begin{lemma}\label{mti stable}
Let $k$ be a field, and let $X$ be a finite-dimensional $k$-vector space equipped with a nondegenerate symplectic form $\omega_X$. Let $\varphi_X$ be a $k$-linear automorphism of finite order $m$ satisfying 
\[
\omega_X(\varphi_X(u), \varphi_X(v))=\omega_X(u, v)\quad\text{for all }u,v\in X. 
\]
    Assume that $x^m-1$ splits completely over $k$. Then $X$ admits a maximal totally isotropic subspace $W'$ such that $\varphi_X(W')=W'$. 
\end{lemma}
\begin{proof}
We prove the assertion by induction on $\dim X$. The case $\dim X=0$ is trivial. 

Assume that $\dim X>0$. Since $x^m-1$ splits completely, $\varphi_X$ admits a nonzero eigenvector $u\in X$. Set $L:=\langle u\rangle$. This is totally isotropic and stable under $\varphi_X$. Therefore, the quotient space 
$L^\perp /L$ carries a nondegenerate symplectic form and linear automorphism induced by $\omega_X$ and $\varphi_X$, respectively. By the induction hypothesis, $L^\perp/L$ admits a maximal totally isotropic subspace that is stable under the induced automorphism. Taking its inverse image in $L^\perp$, we obtain a maximal totally isotropic subspace of $X$ that is stable under $\varphi_X$. 
\end{proof}

As a consequence of Propositions~\ref{2odd impossible} and~\ref{2odd impossible 2}, we obtain the following. 
\begin{corollary}\label{lasc}
Let $m\geq1$ be an odd integer. Assume that $\F_p$ contains a primitive $m$-th root of unity. Then, for any $R$, 
\[
(\mu(\overline{C}_R),\delta(\overline{C}_R))\neq(2m,1). 
\]
\end{corollary}

\appendix

\section{Galois-theoretic criterion for maximality}
\label{Gal}

Let $p$ be a power of a prime number $p_0$. In this section, we also consider the case where $p_0$ is odd. Let $S$ be a normal connected scheme of finite type over $\mathbb{F}_{p}$. Fix elements $a_0,\dots,a_e\in \Gamma(S,{\mathcal O}_S)$ where $e\geq1$ and $a_e$ is invertible. Let $R(x)=\sum_{i=0}^ea_ix^{p^i}$. Consider the relative smooth affine $S$-curve defined by 
\[
C_R \colon y^p-y= xR(x). 
\]

For an $S$-scheme $X$ and a morphism $S'\to S$ of schemes, we write $X_{S'}$ for the base change $X\times_SS'$. 

Let $T\to S$ be a connected finite \'etale Galois covering with Galois group $G$. For a morphism $s\to S$ from the spectrum of a finite field, we write ${\rm Fr}_s$ for 
the conjugacy class in $G$ containing the geometric Frobenius element at $s$.

The following is the main theorem of this section. 
\begin{theorem}\label{T_12}
The following statements hold: 
\begin{enumerate}
\item 
There exists a unique connected finite \'etale Galois covering $T_1\to S$ with Galois group $G_1$ satisfying the following property: For any morphism $s\to S$ from the spectrum of a finite field, 
the curve $C_{R,s}$ is $s$-extremal if and only if 
the conjugacy class ${\rm Fr}_s$ is trivial. In other words, the condition that $C_{R,s}$ is $s$-extremal is equivalent to the existence of a lift $s\to T_1$ of the morphism $s\to S$.  

\item There exists a unique connected finite \'etale Galois covering $T_2\to S$ with Galois group $G_2$ satisfying the following property: For any morphism $s\to S$ from the spectrum of a finite field, 
the curve $C_{R,s}$ is $s$-minimal  if and only if 
the conjugacy class ${\rm Fr}_s$ in $G_2$ is trivial. 
\end{enumerate}
\end{theorem}

\begin{corollary}\label{T2 dominates T1}
 The Galois covering $T_2$ dominates $T_1$. Fix an $S$-morphism $T_2\to T_1$. Let $s\to T_1$ be a morphism from the spectrum of a finite field. We view $s$ as an $S$-scheme by the composite morphism $s\to T_1\to S$. Then the curve $C_{R,s}$ is $s$-maximal if and only if the morphism $s\to T_1$ does not lift to $s\to T_2$. 
\end{corollary}

The remainder of this section is devoted to the proof of Theorem~\ref{T_12}. 
To this end, we recall several constructions associated with the curve $C_R$; in the case where $S$ is the spectrum of a finite field, these are discussed in  \cite[\S2.2]{ITT0} and \cite[\S2.1]{TT}. Since the same constructions work over a general base scheme $S$ without modification, we omit the straightforward verifications.

Define 
\[
 E_R(x):=R(x)^{p^e}+\sum_{i=0}^e(a_ix)^{p^{e-i}}. 
\]
Since $a_e$ is invertible, $E_R$ defines a finite \'etale morphism $E_R\colon \A^1_S\to \A^1_S$. Consider the cartesian diagrams 
\[
\xymatrix{
H_R\ar[d]\ar[r]&V_R\ar[d]\ar[r]&S\ar[d]^-{0_S}\\
C_R\ar[r]^-{\phi}&\A^1_S\ar[r]^-{E_R}&\A^1_S, 
}
\]
where $\phi$ denotes the projection $(x,y)\mapsto x$, and $0_S$ denotes the zero section $S\to \A^1_S$. Since both $\phi$ and $E_R$ are finite \'etale, $V_R$ and $H_R$ are finite \'etale $S$-schemes. 
For an $S$-scheme $T$, we identify 
\begin{align*}
   V_R(T)&=\{a\in\Gamma(T,\mathcal{O}_T)\mid E_R(a)=0\},\\
H_R(T)&=\{(a,b)\in V_R(T)\times \Gamma(T,\mathcal{O}_T)\mid b^p-b=aR(a)\}. 
\end{align*}

We recall that the $S$-schemes $V_R$ and $H_R$ are naturally endowed with additional algebraic structures. Since $E_R$ is $\F_p$-linear, its kernel $V_R$ carries the structure of an $\F_p$-vector space scheme over $S$. Hence, $V_R$ may be identified with a locally constant constructible \'etale sheaf of $\F_p$-vector spaces on $S$. 

Define 
\begin{align*}
    f_R(x,y) &\coloneqq -\sum_{i=0}^{e-1}\left(\sum_{j=0}^{e-i-1}(a_i x^{p^i} y)^{p^j}
+(x  R(y))^{p^i}\right),\\
\omega_R(x,y)&\coloneqq f_R(x,y)-f_R(y,x).
\end{align*}
Then $\omega_R$ defines a symplectic form on $V_R$. Moreover, it is nondegenerate, in the sense that the associated morphism of \'etale sheaves 
\[
V_R\to \mathcal{H}om_{\F_p}(V_R,\F_p)
\]
is an isomorphism. When $S$ is the spectrum of a field, this is verified in \cite[Lemma~2.6]{Tsu}, and the general case follows from this by taking geometric stalks.

For $(a,b),(c,d)\in H_R$, define 
\[
(a,b)\cdot (c,d):=(a+c,b+d+f_R(a,c)). 
\]
One checks that this defines a group law on $H_R$, so that $H_R$ becomes a group scheme over $S$. The commutator is computed as 
\[
[(a,b), (c,d)]=(0,\omega_R(a,c)). 
\]

By using the nondegeneracy of $\omega_R$, it follows that the center of $H_R$ is 
\[
\F_p\cong\{0\}\times\F_p\subset H_R. 
\]

Finally, we recall that $H_R$ acts on $C_R$ from the right via 
\[
(x,y)\cdot (a,b):=(x+a,y+b+f_R(x,a)).
\]

\medskip

Assume now that $S={\rm Spec}(\F_q)$ for a finite field extension $\F_q/\F_p$. The group of $\F$-valued points $H_R(\F)$ 
is naturally equipped with an ${\rm Fr}_q$-action. 
Consider the induced semidirect product 
\[
G:=H_R(\F)\rtimes {\rm Fr}_q^\Z. 
\]
Through the right action of $H_R$ on $C_R$, the cohomology group 
\[
H_c^1:=H_c^1(C_{R,\F},\Ql)
\]
becomes a representation of the group $G$. 
For every nontrivial character $\psi:\F_p\to\Ql^\times$, let $V_\psi$ denote the $\psi$-isotypic component of $H^1_c$ as an $H_R(\F)$-representation. Here, we identify $\F_p=Z(H_R(\F))$, as explained above. Since $\F_p$ is in the center of the group $G$, the space $V_\psi$ is stable under the action of $G$, and thus defines a representation 
\[
\rho_\psi\colon G\to GL(V_\psi). 
\]

We have an isomorphism 
\[
V_\psi\cong H^1_c(\A^1_{\F},\mathcal{L}_\psi(xR(x))). 
\]
This follows from the decomposition 
\[
\phi_\ast\Ql\cong \bigoplus_{\psi\in\F_p^\vee}\mathcal{L}_\psi(xR(x))
\]
and the fact that $\mathcal{L}_\psi(xR(x))$ is the $\psi$-isotypic component of $\phi_\ast\Ql$. Consequently, by \cite[Lemma~3.3]{TT}, we have $\dim V_\psi=p^e$.

We need the following result. 
\begin{lemma}\label{H_R splits}
    Assume that $S={\rm Spec}(\F_q)$ for a finite field extension $\F_q/\F_p$. Then the following hold: 
    \begin{enumerate}
        \item The $H_R(\F)$-representation $V_\psi$ is irreducible. The group homomorphism 
        \[
        \rho_\psi|_{H_R(\F)}\colon H_R(\F)\to GL(V_\psi)
        \]
is injective. 
\item Assume that $C_R$ is $\F_q$-extremal. Then the \'etale covering $H_R\to S$ splits completely. We have $\F_{p^2}\subset \F_q$. 
    \end{enumerate}
\end{lemma}
\begin{proof}
(i) We first prove irreducibility. We apply the Stone--von Neumann theorem for a finite two-step nilpotent group; for this theorem, we refer to \cite[Exercises~4.1.4--4.1.8]{Bum}. 

Since the character $\psi$ is nontrivial, it is generic in the sense of loc.~cit. Therefore, by \cite[Exercise~4.1.8]{Bum}, there exists a unique irreducible representation $\pi_\psi$ of $H_R(\F)$ whose central character is $\psi$. Moreover, its dimension is equal to $p^e$.

Since the center $\F_p$ acts on $V_\psi$ via $\psi$,  the uniqueness of $\pi_\psi$ implies that 
\[
V_\psi\cong \pi_\psi^{\oplus n}
\]
 for some $n\in \Z_{\geq0}$. Since both $V_\psi$ and $\pi_\psi$ have dimension $p^e$, it follows that $n=1$. Hence, $V_\psi$ is irreducible. 
 
We now prove that $\rho_\psi$ is injective on $H_R(\F)$. Let $g=(a,b)\in H_R(\F)$ lie in the kernel. If $a=0$, then $b\in \F_p$, and hence $g\in Z(H_R(\F))$. Thus, for each nontrivial character $\psi$, the element  $g$ acts on $V_\psi$ by scalar multiplication by $\psi(b)$. Since $\rho_\psi(g)=1$, it follows that $\psi(b)=1$. As $\psi$ is arbitrary, we conclude that $b=0$, and hence $g=1$.  

Assume now that $a\neq0$. By the nondegeneracy of $\omega_R$, there exists an element $h\in H_R(\F)$ such that 
\[
[g,h]=(0,c)\qquad \text{for some }c\in\F_p\setminus\{0\}. 
\]
Choose a nontrivial character $\psi$ such that $\psi(c)\neq1$. Then $[g,h]$ acts nontrivially on $V_\psi$. In particular, $\rho_\psi(g)\neq1$. This proves injectivity. 

 \medskip 

 \noindent
 (ii) We first prove that $H_R$ splits. It suffices to show that ${\rm Fr}_q$ commutes with $H_R(\F)$ in $G$. 
 
By the assumption that $C_R$ is $\F_q$-extremal, the image $\rho_\psi({\rm Fr}_q)$ is a scalar matrix. In particular, it commutes with $\rho_\psi(H_R(\F))$. On the other hand, by~(i), the representation $\rho_\psi$ is injective on $H_R(\F)$. Therefore, ${\rm Fr}_q$ commutes with $H_R(\F)$ in $G$. This proves that $H_R$ splits. 

We now show that $\F_{p^2}\subset \F_q$. Since $H_R$ splits, the intermediate covering $V_R$ also splits. Thus, the case where $p_0=2$ is proved in Lemma~\ref{even}. Assume that $p_0$ is odd. In this case,  \cite[Corollary~A.8]{TT} implies that, for each nontrivial character $\psi$, the value 
\[
-\zeta_{p,\psi} \cdot \left(\frac{c}{q}\right)\cdot  G_\psi, \qquad(G_\psi:=\sum_{x\in \F_q}\psi(x^2))
\]
appears as an ${\rm Fr}_q$-eigenvalue of $C_R$, where $c\in \F_q^\times$ is independent of $\psi$, and $\zeta_{p,\psi}$ is a certain $p$-th root of unity. 
Since $C_R$ is $\F_q$-extremal, all ${\rm Fr}_q$-eigenvalues on $H_c^1$ are equal. For every $a\in \F_q^\times$, let $\psi_a$ denote the character $x\mapsto \psi(ax)$. Then 
\[
\zeta_{p,\psi} G_\psi=\zeta_{p,\psi_a} G_{\psi_a}=\zeta_{p,\psi_a}\left(\frac{a}{q}\right) G_{\psi}, 
\]
and therefore $\zeta_{p,\psi}=\zeta_{p,\psi_a}(\frac{a}{q})$. Taking $p$-th powers, we obtain $(\frac{a}{q})=1$ for all $a\in\F_q^\times$. Thus every element of $\F_p^\times$ is a square in $\F_q$, which implies that $\F_{p^2}\subset\F_q$. 
\end{proof}

\medskip 

\noindent\textit{(Proof of Theorem~\ref{T_12})}

{The uniqueness of $T_1$ and $T_2$ follows from the Chebotarev density theorem. We prove the existence by constructing them explicitly. 

\medskip 

\noindent
By Lemma~\ref{H_R splits}, for any morphism $s\to S$ from the spectrum of a finite field such that $C_s$ is $s$-extremal, the finite \'etale $s$-scheme $H_{R,s}$ splits completely. 
Therefore, the Galois coverings $T_1$ and $T_2$ with the required properties should trivialize the finite \'etale $S$-scheme $H_R$. 

Let $S_1\to S$ be the 
smallest connected Galois covering that trivializes $H_R$. Replacing $S_1$ by $S$, we may assume that $H_R$ is trivial over $S$. 
Since $S$ is connected, the $S$-scheme $H_{R}$ can be identified with the constant group scheme associated with the group  $H_R(S)$. 

Since $V_R\to S$ is an intermediate covering of $H_R\to S$, $V_R$ is also trivial. We also identify it with the constant $S$-scheme associated with the $\F_p$-vector space $V_R(S)$. By abuse of notation, 
we write $V_R$ and $H_R$ for $V_R(S)$ and $H_R(S)$, respectively. 

\medskip 

One checks that the action of $H_R$ on $C_R$ is free and that the quotient morphism $C_{R}\to C_{R}/H_R$ is identified with  $E_R\circ\phi\colon C_{R}\to\A^1_{S}$. In the case where $S$ is the spectrum of a finite field, this is proved in \cite[Lemma~2.7]{TT}. The general case follows by the same argument.

Fix a maximal totally isotropic subspace $\overline{A}\subset V_R$ and set $A:=\{(a,b)\in H_R\mid a\in \overline{A}\}$, which is a maximal abelian subgroup of $H_R$. Put $F_A(x):=\prod_{a\in \overline{A}}(x-a)$. The quotient morphism $C_{R}\to C_{R}/A$ is identified with $F_{A}\circ\phi\colon C_{R}\to \A^1_{S}$.

Let $f\colon C_{R}\to S$ denote the structure morphism. We compute $R^1f_!\Ql$ as follows. First we have 
\[\phi_!\Ql\cong \Ql\oplus\bigoplus_{\psi\in\F_p^\vee\setminus\{1\}}{\mathcal L}_{\psi}(xR(x)). \]
Let $g\colon \A^1_{S}\to S$ denote the structure morphism. Since 
$R^1 g_! \Ql=0$, we have 
\[R^1f_!\Ql\cong \bigoplus_{\psi\in\F_p^\vee\setminus\{1\}}R^1g_!({\mathcal L}_{\psi}(xR(x)))=\bigoplus_{\psi\in\F_p^\vee\setminus\{1\}}R^1g_!\mathcal{M}_\psi, \]
where we set $\mathcal{M}_\psi:={\mathcal L}_{\psi}(xR(x))$. 
We claim that $R^1g_!\mathcal{M}_\psi$ is a smooth $\Ql$-sheaf on $S$. Indeed, consider the canonical compactification 
\[
\A^1_S\xhookrightarrow{j}\mathbb{P}^1_S\xrightarrow{\overline{g}}S. 
\]
Then 
\[
R^1g_!\mathcal{M}_\psi\cong R^1\overline{g}_!j_!\mathcal{M}_\psi.
\]
Consequently, by \cite[Corollaire~2.1.2]{Lau} applied to $j_!\mathcal{M}_\psi$, smoothness follows once we check that the function 
\[
S\to\Z,\quad s\mapsto {\rm Sw}_\infty(j_!\mathcal{M}_\psi|_{\mathbb{P}^1_s}\bigr)
\]
is constant, where ${\rm Sw}_\infty$ denotes the Swan conductor at the infinity. This conductor is equal to $p^{e}+1$ by \cite[Lemma~3.3]{TT}, and hence constant. 

Hence, the sheaf $R^1g_!\mathcal{M}_\psi$ is smooth. Moreover, it carries an action of $H_R$. 
Take and fix a geometric point $\eta\to S$ lying over a closed point. Then the stalk 
\[
V_\psi:=(R^1g_!\mathcal{M}_\psi)_\eta
\]
becomes a continuous $\Ql$-representation of the group $\pi_1(S,\eta)\times H_R$. By the proper base change theorem, $V_\psi$ is isomorphic to 
\[
H^1_c(\A^1_\eta,\mathcal{L}_\psi(xR(x))|_{\A^1_\eta}). 
\]
Thus, $V_\psi$ is irreducible as an $H_R$-representation by Lemma~\ref{H_R splits}(i). Since the actions of $\pi_1(S,\eta)$ and $H_R$ commute,  the group $\pi_1(S,\eta)$ acts on $V_\psi$ through a character 
\[\chi_\psi\colon \pi_1(S,\eta)\to\Ql^\times. \]

Using these characters, we define the $S$-scheme $T_1$ to be the Galois covering corresponding to the intersection of the kernels of the characters 
    \[
    \chi_\psi\cdot\chi_{\psi'}^{-1}\qquad\text{for all }\psi,\psi'\in \F_p^\vee\setminus\{1\}. 
    \]

Similarly, we define $T_2$ to be the Galois covering corresponding to the intersection of the kernels of the characters 
    \[
    \Ql(-1/2)\cdot\chi_\psi^{-1}\qquad\text{for all }\psi\in\F_p^\vee\setminus\{1\}, 
    \]
    where $\Ql(-1/2)$ denotes the 
character defined by the composite 
\[\pi_1(S)^{\mathrm{ab}} \to \pi_1({\rm Spec}(\F_{p^2}))\to\Ql^\times.\]
Here, the first map is induced by the structure morphism $S\to{\rm Spec}(\F_{p^2})$, whose existence follows from Lemma~\ref{H_R splits} and the Chebotarev density theorem, and the 
second map sends the geometric Frobenius element to $p$. 
    
To verify that these definitions indeed give \emph{finite} coverings, it suffices to show that the displayed characters have finite order. For this, it is enough to prove that the character 
\[
\theta_\psi\coloneqq \Ql(-1/2)\cdot \chi_\psi^{-1}
\]
has finite order. In fact, we show that $\theta_\psi^{4p_0}$ is trivial. Since this can be checked at closed points, we may assume that $S={\rm Spec}(\F_q)$ for some finite field extension $\F_q/\F_p$. 

This follows from \cite[Theorem 1.2(1)]{TT}, and the proof is complete.
 \qed}

\subsection*{Acknowledgement} 
T. I. is supported by JSPS KAKENHI Grant Numbers 23K20786, 24K21512 and 25K00905.\\
D. T. is supported by JSPS KAKENHI Grant Number 25KJ0122.\\
T. T. is supported by JSPS KAKENHI Grant Numbers 25K06959 and 23K20786.


\begin{thebibliography}{99}
\bibitem{BP}
R.~Blache and T.~Pierre,
\textit{Zeta functions of quadratic Artin--Schreier curves in characteristic two},
Acta Arith.\ \textbf{207} (2023), No.~1, 39--56.

\bibitem{Bum}
D.~W.\ Bump, 
\textit{Automorphic forms and representations}, 
Cambridge Studies in Advanced Mathematics, 55, Cambridge Univ.\ Press, Cambridge, 1997.

\bibitem{C}
R.\ S.\ Coulter,
\textit{The number of rational points of a class of Artin--Schreier curves},
Finite Fields Appl.\ \textbf{8} (2002), No.~4, 397--413.


\bibitem{Go}
D. Goss,
\textit{Basic Structures of Function Field Arithmetic},
Springer, 1996.

\bibitem{GV}
G.\ van der Geer and M.\ van der Vlugt,
\textit{Reed--Muller codes and supersingular curves. I},
Compositio Math.\ \textbf{84} (1992), No.~3, 333--367.


\bibitem{ITT}
T.\ Ito, R.\ Tatematsu and T.\ Tsushima,
\textit{Criteria of maximality and minimality of van der Geer--van der Vlugt curves},
Finite Fields Appl.\ \textbf{111} (2026), Article ID 102781.

\bibitem{ITT0}
T.\ Ito, D.\ Takeuchi and T.\ Tsushima,
\textit{The $L$-polynomials of van der Geer--van der Vlugt curves in characteristic $2$},
J.\ Lond.\ Math.\ Soc.\ \textbf{113} (2026), No.~3, Article ID e70523.

\bibitem{KP}
V.\ Karemaker and R.\ Pries, 
\textit{Fully maximal and fully minimal abelian varieties}, 
J.\ Pure Appl.\ Algebra \textbf{223} (2019), No.~7, 3031--3056.


\bibitem{Lau}
G.\ Laumon, 
\textit{Semi-continuit\'e{} du conducteur de Swan (d'apr\`es P.\ Deligne)}, in {\it The Euler-Poincar\'e{} characteristic}, pp.\ 173--219, Ast\'erisque, 82--83, Soc.\ Math.\ France, Paris. 


\bibitem{St}
H.\ Stichtenoth,
\textit{Algebraic Function Fields and Codes},
2nd ed., Graduate Texts in Mathematics, vol.\ 254,
Springer, 2009.

\bibitem{SX}
H.\ Stichtenoth and C.\ Xing, 
\textit{On the structure of the divisor class group of a class of curves over finite fields},
Arch.\ Math.\ (Basel) \textbf{65} (2) (1995), 141--150. 

\bibitem{TT}
D.\ Takeuchi and T.\ Tsushima,
\textit{Gauss sums and van der Geer--van der Vlugt curves},
Bull.\ Lond.\ Math.\ Soc.\ \textbf{56} (2024), No.~2, 602--623.

\bibitem{Tsu}
T.\  Tsushima, 
\textit{Local Galois representations associated to additive
polynomials}, 
Manuscr.\ Math.\ \textbf{174} (2024), No.~3--4, 1151--1182. 

\end{thebibliography}
\end{document}